\documentclass[11pt]{amsart}
\usepackage{amsmath}
\usepackage{amssymb}
\usepackage{amscd}
\usepackage{enumerate}
\usepackage{verbatim}
\usepackage{color}
\input xy
\xyoption{all}
\newtheorem{Thm}{Theorem}[section]
\newtheorem{Lemma}[Thm]{Lemma}
\newtheorem{Cor}[Thm]{Corollary}
\newtheorem{Prop}[Thm]{Proposition}
\theoremstyle{definition}
\newtheorem{Def}[Thm]{Definition}
\newtheorem{Example}[Thm]{Example}

\newtheorem{Rmk}[Thm]{Remark}

\def\bbz{\mathbb{Z}}
\def\bbc{\mathbb{C}}
\def\bbr{\mathbb{R}}
\def\bbz{\mathbb{Z}}
\def\bbn{\mathbb{N}}
\def\bbq{\mathbb{Q}}

\def\calR{\mathcal{R}}
\def\frakG{\mathfrak{G}}
\def\frakS{\mathfrak{S}}

\def\lra{\longrightarrow}

\def\x{\times}
\def\bs{\backslash}

\def\aut{\mathrm{Aut}}
\def\aff{\mathrm{Aff}}

\def\fix{\mathrm{Fix}}
\def\gfix{\mathrm{fix}}
\def\Nil{\mathrm{Nil}}

\def\id{\mathrm{id}}
\def\tr{\mathrm{tr}}

\def\ind{\mathrm{ind}}

\def\Endo{\mathrm{Endo}}

\def\R{(\mathrm{R})}
\def\E{(\mathrm{E})}
\def\Ad{\mathrm{Ad}}

\def\sp{\mathrm{sp}}

\def\GL{\mathrm{GL}}

\def\bbs{\mathbb{S}}

\def\Sol{\mathrm{Sol}}
\def\GammaA{\Gamma_{\!A}}

\def\sgn{\mathrm{sign~\!}}

\def\coin{\mathrm{coin}}

\def\sol{\mathfrak{sol}}

\def\boxit#1{\vbox{\hrule\hbox{\vrule\kern3pt
     \vbox{\kern3pt#1\kern3pt}\kern3pt\vrule}\hrule}}

\DeclareMathOperator{\grow}{Growth}

\begin{document}
\title[The Nielsen and Reidemeister numbers on infra-solvmanifolds]
{The Nielsen and Reidemeister numbers of maps on infra-solvmanifolds of type $\R$}
\author{Alexander Fel'shtyn }
\address{Instytut Matematyki, Uniwersytet Szczecinski,
ul. Wielkopolska 15, 70-451 Szczecin, Poland  and}

\address{Institut des Hautes \'Etudes Scientifiques, Le Bois-Marie 35, route de Chartres 91440 Bures-sur-Yvette, France}

\email{felshtyn@ihes.fr, fels@wmf.univ.szczecin.pl}

 \author{Jong Bum Lee}
\address{Department of mathematics, Sogang University, Seoul 121-742, KOREA}
\email{jlee@sogang.ac.kr}

\date{\today}

\thanks{The second-named author is partially supported by Basic Science Researcher Program through the National Research Foundation of Korea funded by the Ministry of Education (No. 2013R1A1A2058693) and by the Sogang University Research Grant of 2010(10022).}
\thanks{Tel: +48-91-444-1271; Fax: +48-91-444-1226 (A. Fel'shtyn)}
\thanks{Tel: +82-2-705-8414; Fax: +82-2-714-6284 (J.B.Lee)}


\begin{abstract}
We prove the rationality, the functional equations and calculate the radii of convergence of the Nielsen and the Reidemeister zeta functions of continuous maps on infra-solvmanifolds of type $\R$. We
find a connection between the Reidemeister and Nielsen zeta functions and the Reidemeister torsions of the corresponding  mapping tori. We  show that if  the Reidemeister zeta function is defined for a homeomorphism on an infra-solvmanifold of type $\R$, then this manifold is an infra-nilmanifold.
We also prove that a map on  an infra-solvmanifold of type $\R$ induced by an affine map  minimizes the topological entropy in its  homotopy class and it has a rational Artin-Mazur zeta function.
Finally we prove the Gauss congruences for the Reidemeister and Nielsen numbers of any map on an infra-solvmanifolds of type $\R$ whenever all the Reidemeister numbers of iterates of the map are finite.
Our  main technical tool is the averaging formulas for the Lefschetz, the Nielsen  and the Reidemeister numbers on infra-solvmanifolds of type $\R$.
\end{abstract}
\maketitle

\tableofcontents

\setcounter{section}{-1}

\section{Introduction}

We assume everywhere $X$ to be a connected, compact polyhedron and $f:X\rightarrow X$ to be a continuous map. Let $p:\tilde{X}\rightarrow X$ be the universal cover of $X$ and $\tilde{f}:\tilde{X}\rightarrow \tilde{X}$ a lifting
of $f$, {i.e.,} $p\circ\tilde{f}=f\circ p$. Two lifts $\tilde{f}$ and $\tilde{f}^\prime$ are called \emph{conjugate} if there is a $\gamma\in\Gamma\cong\pi_1(X)$ such that $\tilde{f}^\prime = \gamma\circ\tilde{f}\circ\gamma^{-1}$. The subset $p(\fix(\tilde{f}))\subset \fix(f)$ is called the \emph{fixed point class} of $f$ determined by the lifting class $[\tilde{f}]$. A fixed point class is called \emph{essential} if its index is nonzero. The number of lifting classes of $f$ (and hence the number of fixed point classes, empty or not) is called the \emph{Reidemeister number} of $f$, denoted {by} $R(f)$. This is a positive integer or infinity. The number of essential fixed point classes is called the \emph{Nielsen number} of $f$, denoted by $N(f)$ \cite{Jiang}.

The Nielsen number is always finite. $R(f)$ and $N(f)$ are homotopy invariants. In the category of compact, connected polyhedra the Nielsen number of a map is, apart from in certain exceptional cases, equal to the least number of fixed points of maps with the same homotopy type as $f$.

Let $G$ be a group and $\phi:G\rightarrow G$ an endomorphism. Two elements $\alpha, \alpha^\prime\in G$ are said to be \emph{ $\phi$-conjugate} if and only if there exists $\gamma \in G$ such that $\alpha^\prime=\gamma\alpha\phi(\gamma)^{-1}$.
The number of $\phi$-conjugacy classes is called the \emph{Reidemeister number} of $\phi$, denoted by $R(\phi)$. This is a positive integer or infinity.

Taking a dynamical point of view, we consider the iterates of $f$ and $\phi$, and we may define following \cite{Fel84, PilFel85, Fel88,  Fel91} several zeta functions connected with {the} Nielsen fixed point theory.
The Reidemeister zeta functions of $f$ and $\phi$ and the Nielsen zeta function of $f$ are defined as power series:
\begin{align*}
R_\phi(z)&=\exp\left(\sum_{n=1}^\infty \frac{R(\phi^n)}{n}z^n\right),\\
R_f(z)&=\exp\left(\sum_{n=1}^\infty \frac{R(f^n)}{n}z^n\right),\\
N_f(z)&=\exp\left(\sum_{n=1}^\infty \frac{N(f^n)}{n}z^n\right).
\end{align*}
Whenever we mention the Reidemeister zeta function $R_f(z)$, we shall assume that it is well-defined and so $R(f^n)<\infty$ and $R(\phi^n)<\infty$ for all $n>0$. Hence $R_f(z)=N_f(z)$ on infra-nilmanifolds by  Theorem~\ref{AV-all} below and on infra-solvmanifolds of type $\R$ by Corollary~\ref{R-fix}. However, there are  spaces and maps for which  $R_f(z)$ is not defined. The zeta functions $R_f(z)$ and $N_f(z)$ are homotopy invariants. {The function $N_f(z)$ has a positive radius of convergence for any continuous map $f$ \cite{PilFel85}.} The above zeta functions are directly analogous to the Lefschetz zeta function
$$
L_f(z) := \exp\left(\sum_{n=1}^\infty \frac{L(f^n)}{n} z^n \right),
$$
where
\begin{equation*}\label{Lef}
 L(f^n) := \sum_{k=0}^{\dim X} (-1)^k \tr\Big[f_{*k}^n:H_k(X;\bbq)\to H_k(X;\bbq)\Big]
\end{equation*}
is the Lefschetz number of the iterate $f^n$ of $f$. The Lefschetz zeta function is a rational function of $z$ and is given by the formula:
$$
L_f(z) = \prod_{k=0}^{\dim X}
          \det\big({I}-f_{*k}.z\big)^{(-1)^{k+1}}.
$$

The following problem was investigated: for which spaces and maps and for which groups and endomorphisms are the Nielsen and Reidemeister zeta functions rational functions? Are these functions algebraic functions?

The knowledge that a zeta function is a rational function is important because it shows that the infinite sequence of coefficients of the corresponding power series is closely interconnected, and is given by the finite set of zeros and poles of the zeta function.

In \cite{Fel91, FelHil, fhw, Li94, Fel00}, the rationality of the Reidemeister zeta function $R_\phi(z)$ was proven in the following cases: the group is finitely generated and an endomorphism is eventually commutative; the group is finite; the group is a direct sum of a finite group and a finitely generated free Abelian group; the group is finitely generated, nilpotent and torsion free.
In \cite[Theorem 4]{Wong01} the rationality of the Reidemeister and Nielsen zeta functions was proven for infra-nilmanifold under some (rather technical) sufficient conditions.
It is also known that the Reidemeister numbers of the iterates of an automorphism  of an {almost polycyclic group} satisfy remarkable Gauss congruences \cite{crelle, ft}.

In this paper we investigate the Reidemeister and the Nielsen zeta functions on infra-solvmanifolds of type $\R$. Our  main technical tool is the averaging formulas for the Lefschetz numbers, the Nielsen numbers and the Reidemeister numbers on infra-nilmanifolds and on infra-solvmanifolds of type $\R$.

 Recently, using these averaging formulas, K. Dekimpe and G.-J. Dugardein \cite{DeDu,Du} calculated the  Nielsen numbers  via  Lefschetz numbers and proved the rationality of the Nielsen zeta functions on infra-nilmanifolds.

   We prove in this paper the rationality, the functional equations and calculate the radii of convergence of the Nielsen and the Reidemeister zeta functions of continuous maps on infra-solvmanifolds of type $\R$. We
find a connection between the Reidemeister and Nielsen zeta functions and the Reidemeister torsions of the corresponding  mapping tori. We  show that if  the Reidemeister zeta function is defined for a homeomorphism on an infra-solvmanifold of type $\R$, then this manifold is an infra-nilmanifold.
We also prove that a map on  an infra-solvmanifold of type $\R$ induced by an affine map  minimizes the topological entropy in its  homotopy class and it has a rational Artin-Mazur zeta function. Finally we prove the Gauss congruences for the Reidemeister and Nielsen numbers of any map on an infra-solvmanifolds of type $\R$ whenever all the Reidemeister numbers of iterates of the map are finite.

Let us present the contents of the paper in more details.
  In Section~\ref{DeDu} we describe the averaging formulas for the Lefschetz numbers, the Nielsen numbers and the Reidemeister numbers on infra-nilmanifolds and Dekimpe-Dugardein's formula for the Nielsen numbers.
In Section~\ref{Coin-S}, we obtain a partial generalization   of K. Dekimpe and G.-J. Dugardein's formula   from fixed points on infra-nilmanifolds to coincidences on infra-solvmanifolds of type $\R$ when the holonomy group is a cyclic group.
The rationality and the functional equations for the Reidemeister and the Nielsen zeta functions on infra-solvmanifolds of type $\R$ are proven in Sections~\ref{Rationality} and ~\ref{jiang}.
After studying the asymptotic Nielsen numbers on infra-solvmanifolds of type $\R$ in Section~\ref{Asymptotic}, we discuss the relationship between the topological entropies, the asymptotic Nielsen numbers and the radius of convergence of the Nielsen and the Reidemeister zeta functions in Section~\ref{EC}.
We also prove in Section~\ref{EC} that a map on  an infra-solvmanifold of type $\R$ induced by the affine map  minimizes the topological entropy in its  homotopy class .
In Section~\ref{Tor}, we find a connection between the Nielsen and the Reidemeister zeta functions and the Reidemeister torsions of the corresponding mapping tori. In Section~\ref{jiang}, we obtain the averaging formula for the Reidemeister numbers on infra-solvmanifolds of type $\R$ and we are able to show that the Reidemeister zeta functions on infra-solvmanifolds of type $\R$ coincide with the Nielsen zeta functions. In Section~\ref{No R}, we show
that if  the Reidemeister zeta function is defined for a homeomorphism on an infra-solvmanifold of type $\R$, then this  manifold is an infra-nilmanifold.
In Section~\ref{AM} we prove that the Artin- Mazur zeta function coincides with the Nielsen zeta function and is  a rational function with  functional equation  for  a continuous map on an infra-solvmanifold of type $\R$ induced by an affine map. In Section~\ref{Gauss cong} we prove the Gauss congruences for the Reidemeister and Nielsen numbers of any map on an infra-solvmanifolds of type $\R$ whenever all the Reidemeister numbers of iterates of the map are finite.

\smallskip

\noindent
\textbf{Acknowledgments.}
The first author is indebted to the Institut des Hautes \'Etudes Scientifiques
(Bures-sur-Yvette) for the support and hospitality and the
possibility of the present research during his visit there.
The authors are grateful to Karel  Dekimpe and Gert-Jan Dugardein for helpful comments and
valuable discussions. {The authors would like to thank the referee for making careful corrections to a few expressions and suggesting some relevant references in the original version of the article. This helped improving some results.}

\section{Averaging formulas and Dekimpe-Dugardein's formula}\label{DeDu}

We consider almost Bieberbach groups $\Pi\subset G\rtimes \aut(G)$, where $G$ is a connected, simply connected nilpotent Lie group, and infra-nilmanifolds $M=\Pi\bs{G}$. {It is known that these are exactly the class of almost flat Riemannian manifolds \cite{Ruh}.} It is L. Auslander's result (see, for example, \cite{LR}) that $\Gamma:=\Pi\cap G$ is a lattice of $G$, and is the unique maximal normal nilpotent subgroup of $\Pi$. The group $\Phi=\Pi/\Gamma$ is the \emph{holonomy group} of $\Pi$ or $M$. Thus we have the following commutative diagram:
$$
\CD
1@>>>G@>>>G\rtimes\aut(G)@>>>\aut(G)@>>>1\\
@.@AAA@AAA@AAA\\
1@>>>\Gamma@>>>\Pi@>{p}>>\Phi@>>>1
\endCD
$$
Thus $\Phi$ sits naturally in $\aut(G)$. Denote $\rho:\Phi\to\aut(\frakG)$, $A\mapsto A_*= \text{the differential of $A$}$.

Let $M=\Pi\bs{G}$ be an infra-nilmanifold. Any continuous map $f:M\to M$ induces a homomorphism $\phi:\Pi\to\Pi$. Due to \cite[Theorem~1.1]{KBL95}, we can choose an affine element $(d,D)\in G\rtimes \Endo(G)$ such that
\begin{equation}\label{KBL}
\phi(\alpha)\circ (d,D)=(d,D)\circ\alpha,\quad \forall\alpha\in\Pi.
\end{equation}
This implies that the affine map $(d,D):G\to G$ induces a continuous map on the infra-nilmanifold $M=\Pi\bs{G}$, which is homotopic to $f$. That is, $f$ has an affine homotopy lift $(d,D)$.

By \cite[Lemma~3.1]{LL-JGP}, we can choose a fully invariant subgroup $\Lambda\subset\Gamma$ of $\Pi$ which is of finite index. Therefore $\phi(\Lambda)\subset\Lambda$ and so $\phi$ induces the following commutative diagram
\begin{equation*}
\CD
1@>>>\Lambda@>>>\Pi@>>>\Psi@>>>1\\
@.@VV{\phi'}V@VV{\phi}V@VV{\bar\phi}V\\
1@>>>\Lambda@>>>\Pi@>>>\Psi@>>>1
\endCD
\end{equation*}
where $\Psi=\Pi/\Lambda$ is finite. Applying \eqref{KBL} for $\lambda\in\Lambda\subset\Pi$, we see that
$$
\phi(\lambda)=dD(\lambda)d^{-1}=(\tau_d D)(\lambda)
$$
where $\tau_d$ is the conjugation by $d$. The homomorphism $\phi':\Lambda\to\Lambda$ induces a unique Lie group homomorphism $F=\tau_dD:G\to G$, and hence a Lie algebra homomorphism $F_*:\frakG\to\frakG$. On the other hand, since $\phi(\Lambda)\subset\Lambda$, $f$ has a lift $\bar{f}:N\to N$ on the nilmanifold $N:=\Lambda\bs{G}$ which finitely and regularly covers $M$ and has $\Psi$ as its group of covering transformations.

\begin{Thm}[Averaging Formula {\cite[Theorem~3.4]{LL-JGP}, \cite[Theorem~6.11]{HLP}}]\label{AV-all}
Let $f$ be a continuous map on an infra-nilmanifold $\Pi\bs{G}$ with holonomy group $\Phi$. Let $f$ have an affine homotopy lift $(d,D)$ and let $\phi:\Pi\to\Pi$ be the homomorphism induced by $f$. Then we have
\begin{align*}
L(f)&=\frac{1}{|\Phi|}\sum_{A\in\Phi}\det(I-A_*F_*)=\frac{1}{|\Phi|}\sum_{A\in\Phi}\det(I-A_*D_*),\\
N(f)&=\frac{1}{|\Phi|}\sum_{A\in\Phi}|\det(I-A_*F_*)|=\frac{1}{|\Phi|}\sum_{A\in\Phi}|\det(I-A_*D_*)|,\\
R(f)=R(\phi)&=\frac{1}{|\Phi|}\sum_{A\in\Phi}\sigma\left(\det(A_*-F_*)\right)=\frac{1}{|\Phi|}\sum_{A\in\Phi}\sigma\left(\det(A_*-D_*)\right)
\end{align*}
where $\sigma:\bbr\to\bbr\cup\{\infty\}$ is defined by $\sigma(0)=\infty$ and $\sigma(x)=|x|$ for all $x\ne0$.
\end{Thm}

Recently, Dekimpe and Dugardein in \cite{DeDu} showed the following: Let $f:M\to M$ be a continuous map on an infra-nilmanifold $M$. Then the Nielsen number $N(f)$ is either equal to $|L(f)|$ or equal to the expression $|L(f)-L(f_+)|$, where $f_+$ is a lift of $f$ to a $2$-fold covering of $M$. By exploiting the exact nature of this relationship for all powers of $f$, they proved that the Nielsen zeta function $N_f(z)$ is always a rational function.

Let $M=\Pi\bs{G}$ be an infra-nilmanifold with the holonomy group $\Phi$ and let $f:M\to M$ be a continuous map with an affine homotopy lift $(d,D)$.
Let $A\in\Phi$. Then we can choose $g\in G$ so that $\alpha=(g,A)\in\Pi$. Write $\phi(\alpha)=(g',A')$. By (\ref{KBL}), we have $(g',A')(d,D)=(d,D)(g,A) \Rightarrow A'D=DA$. Thus $\phi$ induces a function $\hat\phi:\Phi\to\Phi$ given by $\hat\phi(A)=A'$ so that it satisfies that
\begin{equation}\label{Dekimpe-eq}
\hat\phi(A)D=DA, \quad \hat\phi(A)_*D_*=D_*A_*
\end{equation}
for all $A\in\Phi$.

In what follows, we shall give a brief description of main results in \cite{DeDu}. We can choose a linear basis of $\frakG$ so that $\rho(\Phi)=\Phi_*\subset\aut(\frakG)$ can be expressed as diagonal block matrices
$$
\left[\begin{matrix}\Phi_1&0\\0&\Phi_2\end{matrix}\right]\subset\GL(n_1,\bbr)\x\GL(n_2,\bbr)
\subset\GL(n,\bbr)
$$
and $D_*$ can be written in block triangular form
$$
\left[\begin{matrix}D_1&*\\0&D_2\end{matrix}\right]
$$
where $D_1$ and $D_2$ have eigenvalues of modulus $\le1$ and $>1$, respectively. We can assume $\Phi=\Phi_1\x\Phi_2$. Every element $\alpha\in\Pi$ is of the form $(a,A)\in G\rtimes\aut(G)$ and $\alpha$ is mapped to $A=(A_1,A_2)$. We define
$$
\Pi_+=\{\alpha\in\Pi\mid \det A_2=1\}.
$$
Then $\Pi_+$ is a subgroup of $\Pi$ of index at most $2$. If $[\Pi:\Pi_+]=2$, then $\Pi_+$ is also an almost Bieberbach group and the corresponding infra-nilmanifold $M_+=\Pi_+\bs{G}$ is a double covering of $M=\Pi\bs{G}$; the map $f$ lifts to a map $f_+:M_+\to M_+$ which has the same affine homotopy lift $(d,D)$ as $f$. If $D_*$ has no eigenvalues of modulus $>1$, then for any $A\in\Phi$, $A=A_1$ and in this case we take $\Pi_+=\Pi$. Now, a main result, Theorem~4.4, of \cite{DeDu} is the following:

\begin{Thm}[{\cite{DeDu},Theorem~4.4, when $\Pi=\Pi_+$ see also proof of Theorem \ref{infra}}]
\label{T4.4}

Let $f$ be a continuous map on an infra-nilmanifold $\Pi\bs{G}$ with an affine homotopy lift $(d,D)$.
Then the Nielsen  numbers  of $f^k$ are
$$
N(f^k)= \begin{cases}
(-1)^{p+(k+1)n}L(f^k),&\text{when $\Pi=\Pi_+$;}\\
(-1)^{p+(k+1)n}\left(L(f_+^k)-L(f^k)\right),&\text{when $\Pi\ne\Pi_+$},
\end{cases}
$$
where $p$ be the number of real eigenvalues of $D_*$ which are $>1$ and $n$ be the number of real eigenvalues of $D_*$ which are $<-1$.
\end{Thm}

\begin{Rmk}
1) In \cite[Theorem~4.4]{DeDu}   Nielsen numbers  $N(f^n)$ are expressed in terms of Lefschetz numbers $L(f^n)$ and $L(f_+^n)$ via a table given by  parity of  $n$.  \\
2) The proof of our Theorem~\ref{infra} covers the case when $\Pi=\Pi_+$ in Theorem \ref{T4.4} above  because in this case  $N(f)=|L(f)|$.
\end{Rmk}

\section{Coincidences on infra-solvmanifolds of type $\R$ with a cyclic holonomy group}
\label{Coin-S}

In this section, we will be concerned with a generalization of Theorem~\ref{T4.4} when $k=1$ (that is, $N(f)=|L(f)|$ or $|L(f_+)-L(f)|$) from fixed points on infra-nilmanifolds to coincidences on infra-solvmanifolds of type $\R$. We obtain a partial result for coincidences on infra-solvmanifolds of type $\R$ when the holonomy group is a cyclic group.

Let $S$ be a connected, simply connected solvable Lie group of type $\R$, and let $C$ be a compact subgroup of $\aut(S)$. Let $\Pi\subset S\rtimes C$ be torsion free and discrete which is a finite extension of the lattice $\Gamma=\Pi\cap S$ of $S$. Such a group $\Pi$ is called an SB-\emph{group} modeled on $S$. The quotient space $\Pi\bs{S}$ is called an \emph{infra-solvmanifold} of type $\R$ with holonomy group $\Phi=\Pi/\Gamma$. When $\Pi\subset S$, $\Pi\bs{S}$ is a \emph{special solvmanifold} of type $\R$. Thus the infra-solvmanifold $\Pi\bs{S}$ is finitely and regularly covered by the special solvmanifold $\Gamma\bs{S}$ with the group of covering transformations $\Phi$. For more details, we refer to \cite{LL-Nagoya}.

Let $M=\Pi\bs{S}$ be an infra-solvmanifold of type $\R$ with the holonomy group $\Phi$. Then $\Phi$ sits naturally in $\aut(S)$. Write $\rho:\Phi\to\aut(\frakS)$, $A\mapsto A_*$. Let $f,g:M\to M$ be maps with affine homotopy lifts $(d,D), (e,E):S\to S$, respectively. Then $f$ and $g$ induce homomorphisms $\phi,\psi:\Pi\to\Pi$ by the following rules:
\begin{align*}
\phi(\alpha)\circ(d,D)=(d,D)\circ\alpha,\ \
\psi(\alpha)\circ(e,E)=(e,E)\circ\alpha\ \ \forall\alpha\in\Pi.
\end{align*}
In turn, we obtain functions $\hat\phi, \hat\psi:\Phi\to\Phi$ satisfying
\begin{align*}
\hat\phi(A)D=DA\ \text{ and }\ \hat\psi(A)E=EA\ \ \forall A\in\Phi.
\end{align*}
Thus
\begin{align}\label{*}
\hat\phi(A)_*D_*=D_*A_*\ \text{ and }\ \hat\psi(A)_*E_*=E_*A_*\ \ \forall A\in\Phi.
\end{align}

\bigskip

Recall the following well-known facts from representation theory:

\begin{Thm}[{H. Maschke}]\label{Maschke}
Let $\rho:\Phi\to \GL(n,\bbr)$ be a representation. Then there exist irreducible representations $\rho_i:\Phi\to\GL(n_i,\bbr)$ such that $\rho$ is similar to $\rho_1\oplus\cdots\oplus\rho_s$.
\end{Thm}

\begin{Thm}\label{Schur}
Let $\Phi=\langle A\rangle$ be a {cyclic group} of order $n$ and let $\rho:\Phi\to\GL(m,\bbr)$ be a faithful $\bbr$-irreducible representation. If $n=1$ then $\rho$ is the trivial representation $\rho_{\mathrm{triv}}$. If $n=2$, then $m=1$ and $\rho(A)=-1$. In this case, we denote $\rho$ by $\tau$. If $n>2$, then there exists $k\in\bbz$ such that $\gcd(n,k)=1$ and $\rho$ is similar to the irreducible rotation given by
$$
\Phi\lra\GL(2,\bbr),\ A\longmapsto
\left[\begin{matrix}\cos\frac{2k\pi}{n}&-\sin\frac{2k\pi}{n}\\
\sin\frac{2k\pi}{n}&\hspace{8pt}\cos\frac{2k\pi}{n}\end{matrix}\right].
$$
\end{Thm}

\bigskip

Consider the case where the infra-solvmanifold $M$ of type $\R$ is {orientable} ({for coincidences}) with holonomy group $\Phi$ a {cyclic group} with a generator $A_0$. By Theorem~\ref{Maschke}, the natural representation $\rho:\Phi\to\aut(S)\cong\aut(\frakS)$ is similar to a sum of irreducible representations. If $\sigma:\Phi\to\GL(m,\bbr)$ is irreducible, then the induced representation $\bar\sigma:\Phi/\ker\rho\to\GL(m,\bbr)$ is faithful and irreducible. By Theorem~\ref{Schur}, $\bar\sigma$ is similar to $\rho_{\mathrm{triv}}$, $\tau$ or a rotation. Thus we may assume that $\rho=m\rho_{\mathrm{triv}}\oplus k\tau\oplus\rho_1\oplus\cdots\oplus\rho_t$, where $\rho_i:\Phi\to\GL(2,\bbr)$ is an irreducible rotation. That is, there is a linear basis of $\frakS$ so that $\rho(A_0)\in\aut(\frakS)$ can be represented as diagonal block matrices
$$
\rho(A_0)=\left[\begin{matrix}I_m&&&&\\&-I_k&&&\\&&\Phi_1&&\\&&&\ddots&\\&&&&\Phi_t
\end{matrix}\right]\
\text{ where }
\Phi_i=\rho_i(A_0)\in\GL(2,\bbr).
$$
Remark that if $k>0$ then the {order of $\Phi$ is even}, and $\det(\rho_i(A_0))=1$ for all $i$. Hence $\det(\rho(A_0))=1$ if and only if $k$ is even. This is the only case when the infra-solvmanifold is orientable and hence {$k$ is even}.

Using the identities (\ref{*}), we can write $D_*$ and $E_*$ as block matrices
\begin{align*}
D_*=\left[\begin{matrix}D_{\mathrm{triv}}&0&0\\0&D_\tau&0\\{*}&*&\hat{D}\end{matrix}\right],\quad
E_*=\left[\begin{matrix}E_{\mathrm{triv}}&0&0\\0&E_\tau&0\\{*}&*&\hat{E}\end{matrix}\right]
\end{align*}
where $D_{\mathrm{triv}}, E_{\mathrm{triv}}$ are $m\x m$, $D_\tau, E_\tau$ are $k\x k$ and $\hat{D}, \hat{E}$ are $2t\x2t$.

For $A\in\Phi$, $A=A_0^p$ for some $p$ and
$$
A_*=\left[\begin{matrix}I_m&&\\&(-1)^p I_k&\\&&A_*\end{matrix}\right].
$$
Write $\hat\rho=\rho_1\oplus\cdots\oplus\rho_t:\Phi\to\GL(2t,\bbr), A\mapsto \hat\rho(A)=A_*$ (abusing the notation: $\rho(A)=A_*$). Then the identities (\ref{*}) induce
\begin{align*}
\hat\phi(A)_*\hat{D}=\hat{D}A_*,\ \hat\psi(A)_*\hat{E}=\hat{E}A_*.
\end{align*}
Hence, for all $A=A_0^p$ and $B=A_0^q\in\Phi$, we have that
\begin{align}\label{gen}
\det&(E_*-A_*D_*)\det(E_*-B_*D_*)\\
&=\det(E_{\mathrm{triv}}-D_{\mathrm{triv}})^2\det(E_\tau-(-1)^pD_\tau)\det(E_\tau-(-1)^qD_\tau)\notag\\
&\qquad\x\det(\hat{E}-A_*\hat{D})\det(\hat{E}-B_*\hat{D}).\notag
\end{align}
Note here that $\det(\hat{E}-A_*\hat{D})\det(\hat{E}-B_*\hat{D})\ge0$, see \cite[Lemma~6.3]{DP}.
From \cite[Theorem~4.5]{HL-a}, immediately we have:

\begin{Thm}[{Compare with \cite[Theorem~6.1]{DP}}]
Let $f$ and $g$ be continuous maps on an orientable infra-solvmanifold of type $\R$ with cyclic holonomy group $\Phi=\langle A_0\rangle$. If $\rho(A_0)$ has no eigenvalue $-1$, i.e., if $k=0$, then $N(f,g)=|L(f,g)|$.
\end{Thm}


{Assume $k>0$ (is even); then $\Phi=\langle A_0\rangle$ is of even order. Let $\Phi_0=\langle A_0^2\rangle$ and let $\Pi_0$ be the subgroup of $\Pi$ induced by the inclusion $\Phi_0\hookrightarrow \Phi$.} Remark also that if $D_\tau=0$ or $E_\tau=0$, then we still have $N(f,g)=|L(f,g)|$. {We also assume that $D_\tau\ne0$ and $E_\tau\ne0$.}

\begin{Thm}
Then $\Pi_0$ is a subgroup of $\Pi$ of index $2$, $\Pi_0$ is also an {\rm SB}-group and the corresponding infra-solvmanifold $M_0=\Pi_0\bs{S}$ is a double covering of $M=\Pi\bs{S}$; {the maps $f,g$ lift to map $f_0, g_0:M_0\to M_0$ which have the same affine homotopy lifts $(d,D), (e,E)$ as $f$ and $g$}.
\end{Thm}

\begin{proof}
It is clear that $[\Pi:\Pi_0]=2$ and that $\Pi_0$ is also an SB-group and the corresponding infra-solvmanifold $\Pi_0\bs{S}$ is a double covering of $\Pi\bs{S}$.

To prove the last assertion, we may {consider and} assume that $(d,D):S\to S$ induces $f$ and that $\phi:\Pi\to\Pi$ is a homomorphism such that
$$
\phi(\alpha)(d,D)=(d,D)\alpha, \quad \forall\alpha\in\Pi.
$$
We need to show that $(d,D)$ also induces a map on $\Pi_0\bs{S}$. For this purpose, it is enough to show that $\phi(\Pi_0)\subset\Pi_0$. For any $\beta=(a,A)\in\Pi_0$, let $\phi(\beta)=(b,\hat\phi(A))$. Since $(a,A)\in\Pi_0$, we have $A\in\Phi_0$. The above identity implies that
\begin{align*}
\hat\phi(A)_*D_*=D_*A_* \Rightarrow D_\tau=0 \text{ or } \hat\phi(A)\in\Phi_0.
\end{align*}
Since $D_\tau\ne0$, this finishes the proof of the last assertion.
\end{proof}

For any $A=A_0^p\in\Phi$, we recall from (\ref{gen}) that
\begin{align*}
&\det(E_*-A_*D_*)=\det(E_{\mathrm{triv}}-D_{\mathrm{triv}})\det(E_\tau-(-1)^pD_\tau)\det(\hat{E}-A_*\hat{D})
\end{align*}
and
\begin{align*}
&\det(\hat{E}-\hat{D})\det(\hat{E}-A_*\hat{D})\ge0.
\end{align*}
Let
\begin{align*}
{\epsilon_o=\sgn \det(E_\tau-D_\tau),\
\epsilon_e=\sgn \det(E_\tau+D_\tau).}
\end{align*}
Then $\epsilon_o=\pm\epsilon_e$. {Notice that the values $\epsilon_o$ and $\epsilon_e$ depend both on $f$ and $g$.}  When $\epsilon_o=\epsilon_e$, we still have $N(f,g)=|L(f,g)|$. When $\epsilon_o=-\epsilon_e$, we have that
\begin{align*}
N(f,g)&=\frac{1}{|\Phi|}\sum_{A\in\Phi}|\det(E_*-A_*D_*)|\\
&=\frac{1}{|\Phi|}\left(\sum_{A\in\Phi_0}|\det(E_*-A_*D_*)|+\sum_{A\notin\Phi_0}|\det(E_*-A_*D_*)|\right)\\
&=\frac{\epsilon_o}{|\Phi|}\left(\sum_{A\in\Phi_0}\det(E_*-A_*D_*)
-\sum_{A\notin\Phi_0}\det(E_*-A_*D_*)\right)\\
&=\frac{\epsilon_o}{|\Phi|}\left(2\sum_{A\in\Phi_0}\det(E_*-A_*D_*)
-\sum_{A\in\Phi}\det(E_*-A_*D_*)\right)\\
&=\epsilon_o\left(\frac{1}{|\Phi_0|}\sum_{A\in\Phi_0}\det(E_*-A_*D_*)
-\frac{1}{|\Phi|}\sum_{A\in\Phi}\det(E_*-A_*D_*)\right)\\
&={\epsilon_o(L(f_0,g_0)-L(f,g))}.
\end{align*}

Therefore, we can summarize what we have observed as follows:
\begin{Thm}
Let $M=\Pi\bs{S}$ be an orientable infra-solvmanifold of type $\R$ with cyclic holonomy group $\Phi=\langle A_0\rangle$. Let $\rho:\Phi\to\aut(\frakG)$ be the natural presentation. Then $\rho$ is similar to the sum of irreducible representations $m\rho_{\mathrm{triv}}\oplus k\tau\oplus\rho_1\oplus\cdots\oplus\rho_t$, where $\rho_{\mathrm{triv}}:\Phi\to\GL(1,\bbr)$ is the trivial representation, $\tau:\Phi\to\GL(1,\bbr)$ is the representation given by $\tau(A_0)=-1$, and $\rho_i:\Phi\to\GL(2,\bbr)$ is an irreducible rotation. Let $f,g:M\to M$ be continuous maps with affine homotopy lifts $(d,D),(e,E)$ respectively. Then $D_*$ and $E_*$ can be expressed as block matrices
\begin{align*}
D_*=\left[\begin{matrix}D_{\mathrm{triv}}&0&0\\0&D_\tau&0\\{*}&*&\hat{D}\end{matrix}\right],\quad
E_*=\left[\begin{matrix}E_{\mathrm{triv}}&0&0\\0&E_\tau&0\\{*}&*&\hat{E}\end{matrix}\right]
\end{align*}
where $D_{\mathrm{triv}}, E_{\mathrm{triv}}$ are $m\x m$, $D_\tau, E_\tau$ are $k\x k$ and $\hat{D}, \hat{E}$ are $2t\x2t$. Moreover, we have that:
\begin{enumerate}
\item[$(1)$] If $k=0$, then $N(f,g)=|L(f,g)|$.
\item[$(2)$] If $k>0$ and $\det(E_\tau-D_\tau)\det(E_\tau+D_\tau)\ge0$, then $N(f,g)=|L(f,g)|$.
\item[$(3)$] If $k>0$ and $\det(E_\tau-D_\tau)\det(E_\tau+D_\tau)<0$, then the maps $f,g$ lift to maps $f_0, g_0:M_0\to M_0$ on a double covering $M_0$ of $M$ which have the same homotopy lifts as $f,g$ respectively so that the following formula holds
    $$
    N(f,g)=|L(f_0,g_0)-L(f,g)|.
    $$
\end{enumerate}
\end{Thm}

\begin{proof}
We are left to notice only one thing: If $D_\tau=0$ or $E_\tau=0$, then $k>0$ is even and so $\det(E_\tau-D_\tau)\det(E_\tau+D_\tau)\ge0$.
\end{proof}

\section{The rationality and the functional equation}\label{Rationality}

We start with an example that shows how different can be the Nielsen, the Reidemeister and the Lefschetz zeta functions.

\begin{Example}[\cite{Fel00}]\label{wedge}
Let $f:S^2\vee S^4\rightarrow S^2\vee S^4$ to be a continuous map of the bouquet of spheres such that the restriction $f|_{S^4}=id_{S^4}$ and the degree of the restriction $f|_{S^2}:S^2\rightarrow S^2$ equal to $-2$. Then $L(f)=0$, hence
$N(f)=0$ since $ S^2\vee S^4$ is simply connected.  For $k>1$ we have $L(f^k)=2+(-2)^k\not=0$,  therefore $N(f^k)=1$. $R(f^k)=1$ for all $k\geq 1$ since $ S^2\vee S^4$ is simply connected. From this we have by direct calculation that
\begin{equation*}
N_f(z)=\exp(-z)\cdot \frac{1}{1-z};\  R_f(z)=  \frac{1}{1-z};\ L_f(z)= \frac{1}{(1-z)^2(1+2z)}.
\end{equation*}
Hence $N_f(z)$ is a meromorphic function, and $R_f(z)$ and  $L_f(z) $ are rational and different.
\end{Example}

We give now some other examples of the Nielsen and the Reidemeister zeta functions on infra-nilmanifolds.
For the explicit computation of the zeta functions, the following is useful.
\begin{Prop}\label{RZ}
Let $f$ be a continuous map on an infra-nilmanifold $\Pi\bs{G}$ with holonomy group $\Phi$. Let $f$ have an affine homotopy lift $(d,D)$ and let $\phi:\Pi\to\Pi$ be the homomorphism induced by $f$. Then
\begin{equation*}
\begin{split}
N_f(z)=\prod_{A\in\Phi}\sqrt[|\Phi|]{\exp\left(\sum_{n=1}^\infty\frac{|\det(A_*-D_*^n)|}{n}z^n\right)}.
\end{split}
\end{equation*}
When $R_f(z)$ is defined, $R_f(z)=R_\phi(z)=N_f(z)$.
\end{Prop}

\begin{proof}
We may assume $R_f(z)$ is defined. By Theorem~\ref{AV-all}, we have that $R_f(z)=R_\phi(z)=N_f(z)$ and
\begin{equation*}
\begin{split}
R_\phi(z)&=\exp\left(\sum_{n=1}^\infty\frac{R(\phi^n)}{n}z^n\right)\\
&=\exp\left(\sum_{n=1}^\infty\frac{\frac{1}{|\Phi|}\sum_{A\in\Phi}|\det(A_*-F_*^n)|}{n}z^n\right)\\
&=\prod_{A\in\Phi}\left(\exp\left(\sum_{n=1}^\infty\frac{|\det(A_*-F_*^n)|}{n}z^n\right)\right)^{\frac{1}{|\Phi|}}\\
&=\prod_{A\in\Phi}\sqrt[|\Phi|]{\exp\left(\sum_{n=1}^\infty\frac{|\det(A_*-F_*^n)|}{n}z^n\right)}.\qedhere
\end{split}
\end{equation*}
\end{proof}

\begin{Example}\label{ex1}
This is an example used by Anosov to show that the Anosov relation does not hold when the manifold is not a nilmanifold \cite{Anosov}.

Let $\alpha=(a,A)$ and $t_i=(e_i, I_2)$ be elements of $\bbr^2\rtimes\aut(\bbr^2)$, where
$$
a=\left[\begin{matrix}\tfrac{1}{2}\\0\end{matrix}\right],\
A=\left[\begin{matrix}1&\hspace{8pt}0\\0&-1\end{matrix}\right],\
e_1=\left[\begin{matrix}1\\0\end{matrix}\right],\
e_2=\left[\begin{matrix}0\\1\end{matrix}\right].
$$
Then $A$ has period 2, $(a,A)^2=(a+Aa,I_2)=(e_1,I_2)$, and $t_2\alpha=\alpha t_2^{-1}$. Let $\Gamma$ be the subgroup generated by $t_1$ and $t_2$. Then it forms a lattice in $\bbr^2$ and the quotient space $\Gamma\bs\bbr^2$ is the 2-torus. It is easy to check that the subgroup
$$
\Pi=\langle \Gamma, (a,A)\rangle\subset \bbr^2\rtimes\aut(\bbr^2)
$$
generated by the lattice $\Gamma$ and the element $(a,A)$ is discrete and torsion free. Furthermore, $\Gamma$ is a normal subgroup of $\Pi$ of index 2. Thus $\Pi$ is an (almost) Bieberbach group, which is the Klein bottle group, and the quotient space $\Pi\bs\bbr^2$ is the Klein bottle. Thus $\Gamma\bs\bbr^2\to\Pi\bs\bbr^2$ is a double covering projection.

Let $K:\bbr^2\to\bbr^2$ be the linear automorphism given by
$$
K=\left[\begin{matrix}-1&0\\\hspace{8pt}0&2\end{matrix}\right].
$$
It is not difficult to check that $K$ induces $\bar{f}:\Gamma\bs\bbr^2\to\Gamma\bs\bbr^2$ and $f:\Pi\bs\bbr^2\to\Pi\bs\bbr^2$ so that the following diagram is commutative:
$$
\CD
\bbr^2@>K>>\bbr^2\\
@VVV@VVV\\
\Gamma\bs\bbr^2@>{\bar{f}}>>\Gamma\bs\bbr^2\\
@VVV@VVV\\
\Pi\bs\bbr^2@>{f}>>\Pi\bs\bbr^2
\endCD
$$
Note that all the vertical maps are the natural covering maps. In particular, $\Gamma\bs\bbr^2\to\Pi\bs\bbr^2$ is a double covering by the holonomy group of $\Pi/\Gamma$, which is $\Phi=\{I,A\}\cong\bbz_2$.
By Theorem~\ref{AV-all}, we have
\begin{align*}
L(f^n)&=\frac{1}{2}\left(\det(I-K^n)+\det(I-AK^n)\right)=1-(-1)^n,\\
N(f^n)&=2^n(1-(-1)^n).
\end{align*}
In particular, $R(f^n)=2^{n+1}$ when $n$ is odd; otherwise, $R(f^n)=\infty$.
Therefore, the Reidemeister zeta function $R_{f}(z)$ is not defined, and
\begin{align*}
L_{f}(z)&=\exp\left(\sum_{n=1}^\infty\frac{2}{2n-1}z^{2n-1}\right)=\frac{1+z}{1-z},\\
N_{f}(z)&=\exp\left(\sum_{n=1}^\infty\frac{2^{2n}}{2n-1}z^{2n-1}\right)\\
&=\exp\left(\sum_{n=1}^\infty\frac{2}{2n-1}(2z)^{2n-1}\right)=\frac{1+2z}{1-2z}.\\
\end{align*}
\end{Example}

\begin{Example}\label{ex3}
Consider Example~3.5 of \cite{LL-JGP} in which an infra-nilmanifold $M$ modeled on the 3-dimensional Heisenberg group $\Nil$ has the holonomy group of order $2$ generated by $A$ and a self-map $f$ on $M$ is induced by the automorphism $D:\Nil\to\Nil$ given by
$$
D:\left[\begin{matrix}1&x&z\\0&1&y\\0&0&1\end{matrix}\right]
\longmapsto
\left[\begin{matrix}1&-4x-y&z'\\0&1&6x+2y\\0&0&1\end{matrix}\right]
$$
where $z'=-2z-(12x^2+10xy+y^2)$.
Then with respect to the ordered (linear) basis for the Lie algebra of $\Nil$
$$
{\bf e}_1=\left[\begin{matrix}0&0&1\\0&0&0\\0&0&0\end{matrix}\right],\
{\bf e}_2=\left[\begin{matrix}0&1&0\\0&0&0\\0&0&0\end{matrix}\right],\
{\bf e}_3=\left[\begin{matrix}0&0&0\\0&0&1\\0&0&0\end{matrix}\right],
$$
the differentials of $A$ and $D$ are
$$
A_*=\left[\begin{matrix}1&\hspace{8pt}0&\hspace{8pt}0\\0&-1&\hspace{8pt}0\\0&\hspace{8pt}0&-1\end{matrix}\right],\
D_*=\left[\begin{matrix}-2&\hspace{8pt}0&\hspace{8pt}0\\\hspace{8pt}0&-4&-1
\\\hspace{8pt}0&\hspace{8pt}6&\hspace{8pt}2\end{matrix}\right].
$$
By Proposition~\ref{RZ}, we have
\begin{align*}
R_\phi(z)&=\sqrt{\exp\left(\sum_{n=1}^\infty\frac{|\det(I-D_*^n)|}{n}z^n\right)}\sqrt{\exp\left(\sum_{n=1}^\infty\frac{|\det(A_*-D_*^n)|}{n}z^n\right)}.
\end{align*}
Remark that $A_*$ is a block diagonal matrix with $1\x1$ block $I_1$ and $2\x2$ block $-I_2$. We have
\begin{align*}
|\det(A_*-D_*^n)|&=|\det(I_1-D_1^n)\det(-I_2-D_2^n)|\\
&=|\det(I_1-D_1^n)||\det(I_2+D_2^n)|\\
&=|(1-(-2)^n)|(-1)^n\det(I_2+D_2^n)\\
&=(2^n-(-1)^n)(-1)^n\sum_i\tr({\bigwedge}^{\!i}D_2^n).
\end{align*}
Consequently, we obtain
\begin{align*}
&\exp\left(\sum_{n=1}^\infty\frac{|\det(A_*-D_*^n)|}{n}z^n\right)\\
&=\exp\left(\sum_{n=1}^\infty\frac{(2^n-(-1)^n)(-1)^{n}\sum_i\tr({\bigwedge}^{\!i}D_2^n)}{n}z^n\right)\\
&=\exp\left(\sum_{n=1}^\infty\frac{\sum_i\tr({\bigwedge}^{\!i}D_2^n)}{n}(-2z)^n-\sum_{n=1}^\infty\frac{\sum_i\tr({\bigwedge}^{\!i}D_2^n)}{n}z^n\right)\\
&=\prod_i\frac{\det(I-z{\bigwedge}^{\!i}D_2)}{\det(I+2z{\bigwedge}^{\!i}D_2)}\\
&=\frac{1-z}{1+2z}\cdot\frac{1+2z-2z^2}{1-4z-8z^2}\cdot\frac{1+2z}{1-4z}.
\end{align*}
In a similar fashion, we compute
\begin{align*}
&\exp\left(\sum_{n=1}^\infty\frac{|\det(I-D_*^n)|}{n}z^n\right)\\
&=\exp\left(\sum_{n=1}^\infty\frac{|\det(I_1-D_1^n)\det(I_2-D_2^n)|}{n}z^n\right)\\
&=\exp\left(\sum_{n=1}^\infty\frac{(2^n-(-1)^n)(-1)^{n+1}\sum_i(-1)^{i}\tr({\bigwedge}^{\!i}D_2^n)}{n}z^n\right)\\
&=\prod_i\left(\frac{\det(I+2z{\bigwedge}^{\!i}D_2)}{\det(I-z{\bigwedge}^{\!i}D_2)}\right)^{(-1)^i}\\
&=\frac{1+2z}{1-z}\cdot\frac{1+2z-2z^2}{1-4z-8z^2}\cdot\frac{1-4z}{1+2z}.
\end{align*}
The last identity of the above computations follows from the definition of ${\bigwedge}^{\!i}D_2$ (see \cite[Lemma~3.2]{HLP11}). Namely, we have
$$
{\bigwedge}^{\!0}D_2=1,\ {\bigwedge}^{\!1}D_2=D_2,\ {\bigwedge}^{\!2}D_2=\det(D_2)=-2.
$$
In all, we obtain that
\begin{align*}
&N_f(z)=R_f(z)=\frac{1+2z-2z^2}{1-4z-8z^2}.
\end{align*}
\end{Example}

\begin{Thm}\label{infra}
Let $f$ be a continuous map on an  infra-nilmanifold with an affine homotopy lift $(d,D)$.
Assume $N(f)=|L(f)|$. Then the Nielsen zeta function $N_f(z)$ is a rational function and is equal to
\begin{equation*}
N_f(z)=L_f((-1)^qz)^{(-1)^r}
\end{equation*}
where  $q$ is the number of real eigenvalues of $D_*$ which are $<-1$ and
$r$ is the number of real eigenvalues of $D_*$ of modulus $>1$.
When the Reidemeister zeta function $R_f(z)$ is defined, we have
$$
 R_f(z)=R_\phi(z)=N_f(z)
$$.
\end{Thm}

\begin{proof}
By \cite[Theorem~8.2.2]{P} $N(f)=|L(f)|$ implies $N(f^n)=|L(f^n)|$ for all $n$.
Let $\epsilon_n$ be the sign of $\det(I-D^n_*)$. Let $q$ be the number of real eigenvalues of $D_*$ which are less than $-1$ and $r$ be the number of real eigenvalues of $D_*$ of modulus $>1$. Then $\epsilon_n=(-1)^{r+qn}$.
By Theorem~\ref{AV-all}, we have that $\epsilon_1\det(I-A_*D_*)\ge0$ for all $A\in\Phi$. In particular, we have
$$
\det(I-A_*D_*)\det(I-B_*D_*)\ge0\quad \text{for all $A,B\in\Phi$}.
$$
Choose arbitrary $n>0$. By \cite[Lemma~8.2.1]{P},
$$
\det(I-A_*D_*^n)\det(I-D^n_*)\ge0\quad \text{for all $A\in\Phi$}.
$$
Hence we have $N(f^n)=\epsilon_n L(f^n)=(-1)^{r+qn} L(f^n)$. Consequently,
\begin{align*}
N_f(z)&=\exp\left(\sum_{n=1}^\infty\frac{N(f^n)}{n}z^n\right)\\
&=\exp\left(\sum_{n=1}^\infty\frac{(-1)^{r+qn} L(f^n)}{n}z^n\right)\\
&=\left(\exp\left(\sum_{n=1}^\infty\frac{L(f^n)}{n}((-1)^qz)^n\right)\right)^{(-1)^r}\\ &=L_f((-1)^qz)^{(-1)^r}
\end{align*}
is a rational function.

Assume $R_f(z)$ is defined. So, $R(f^n)=R(\phi^n)<\infty$ for all $n>0$. On infra-nilmanifolds, by Theorem~\ref{AV-all}, it is equivalent to saying that $\det(A_*-D_*^n)\ne0$ for all $A\in\Phi$ and all $n$, and hence $\sigma\left(\det(A_*-D_*^n)\right)=|\det(A_*-D_*^n)|$. Thus
\begin{align*}
R(f^n)=R(\phi^n)&=\frac{1}{|\Phi|}\sum_{A\in\Phi}\sigma\left(\det(A_*-D_*^n)\right)\\
&=\frac{1}{|\Phi|}\sum_{A\in\Phi}|\det(A_*-D_*^n)|=N(f^n).
\end{align*}
This implies that $R_f(z)=R_\phi(z)=N_f(z)$.
\end{proof}

 Therefore, for those classes of maps on infra-nilmanifolds for which Anosov relation $N(f)=|L(f)|$ holds \cite{KL, Malfait, DRM} and for those classes of infra-nilmanifolds for which Anosov relation $N(f)=|L(f)|$ holds for ALL maps \cite{Anosov, DRM-JGP, DRM, DRP}, the Nielsen zeta functions and the Reidemeister zeta functions are rational functions.

In general case, using the results of Theorem~\ref{T4.4},  Dekimpe and  Dugardein described the  Nielsen zeta function of $f$ as follows:

\begin{Thm}[{\cite[Theorem~4.5]{DeDu}}]\label{T4.5}
Let $f$ be a continuous map on an infra-nilmanifold $\Pi\bs{G}$ with an affine homotopy lift $(d,D)$. Then the Nielsen zeta function  is a rational function and is equal to
\begin{equation*}
N_f(z)=\begin{cases}
L_f((-1)^nz)^{(-1)^{p+n}}&\text{when $\Pi=\Pi_+$;}\\
\left(\frac{L_{f_+}((-1)^nz)}{L_f((-1)^nz)}\right)^{(-1)^{p+n}}&\text{when $\Pi\ne\Pi_+$,}
\end{cases}
\end{equation*}
where  $p$ is the number of real eigenvalues of $D_*$ which are $>1$ and $n$ is the number of real eigenvalues of $D_*$ which are $<-1$.

When the Reidemeister zeta function $R_f(z)$ is defined, we have
$$
R_f(z)=R_\phi(z)=N_f(z)
$$.
\end{Thm}

\begin{Rmk}\label{NtoS1}
In \cite[Theorem~4.5]{DeDu} the  Nielsen zeta function is expressed in terms of Lefschetz zeta functions $L_f(z)$ and $L_{f_+}(z)$ via a table given by  parity of $p$ and $n$.
{The class of infra-solvmanifolds of type $\R$ contains and shares a lot of properties of the class of infra-nilmanifolds such as the averaging formula for Nielsen numbers, see \cite{HLP11,LL-Nagoya}. Therefore, Theorem~\ref{T4.4} and the statement about $N_f(z)$ in Theorem~\ref{T4.5} can be generalized directly to the class of infra-solvmanifolds of type $\R$, see Remark in \cite[Sec.~\!4]{DeDu}.}
\end{Rmk}

To write down a functional equation for the Reidemeister and the Nielsen zeta function, we recall the following functional equation for the Lefschetz zeta function:

\begin{Lemma}[{\cite[Proposition~8]{Fri1}}, see also \cite{del}] \label{Fried}
{Let $M$ be a closed orientable manifold of dimension $m$ and let $f:M\to M$ be a continuous map of degree $d$. Then
$$
L_{f}\left(\frac{\alpha}{dz}\right)=\epsilon\,(-\alpha dz)^{(-1)^m\chi(M)}\,L_{f}(\alpha z)^{(-1)^m}
$$
where $\alpha=\pm1$ and $\epsilon\in\bbc$ is a non-zero constant such that if $|d|=1$ then $\epsilon=\pm1$.}
\end{Lemma}
\begin{proof}
In the Lefschetz zeta function formula \eqref{Lef}, we may replace $f_*$ by $f^*:H^*(M;\bbq)\to H^*(M;\bbq)$. Let $\beta_k=\dim H_k(M;\bbq)$ be the $k$th Betti number of $M$. Let $\lambda_{k,j}$ be the (complex and distinct) eigenvalues of ${f_*}_k:H_k(M;\bbq)\to H_k(M;\bbq)$

Via the natural non-singular pairing in the cohomology $H^k(M;\bbq)\otimes H^{m-k}(M;\bbq)\to\bbq$, the operators $f^*_{m-k}$ and $d(f^*_k)$ are adjoint to each other. Hence since $\lambda_{k,j}$ is an eigenvalue of $f^*_k$, $\mu_{\ell,j}=d/\lambda_{k,j}$ is an eigenvalue of $f^*_{m-k}=f^*_\ell$. Furthermore, $\beta_k=\beta_{m-k}=\beta_\ell$.

Consequently, we have
\begin{align*}
L_{f}\left(\frac{\alpha}{dz}\right)&=\prod_{k=0}^{m}\prod_{j=1}^{\beta_k} \left(1-\lambda_{k,j}\frac{\alpha}{dz}\right)^{(-1)^{k+1}}\\
&=\prod_{k=0}^{m}\prod_{j=1}^{\beta_k} \left({1-\frac{d}{\lambda_{k,j}}\alpha z}\right)^{(-1)^{k+1}}\left(-\frac{\alpha dz}{\lambda_{k,j}}\right)^{(-1)^{k}}\\
&=\prod_{\ell=0}^{m}\prod_{j=1}^{\beta_{m-\ell}} \left({1-\mu_{\ell,j}\alpha z}\right)^{(-1)^{m-\ell+1}}\prod_{k=0}^{m}\prod_{j=1}^{\beta_k} \left(-\frac{\alpha dz}{\lambda_{k,j}}\right)^{(-1)^{m-\ell}}\\
&=\left(\prod_{\ell=0}^{m}\prod_{j=1}^{\beta_{\ell}} \left({1-\mu_{\ell,j}\alpha z}\right)^{(-1)^{\ell+1}}\prod_{k=0}^{m}\prod_{j=1}^{\beta_k} \left(-\frac{\alpha dz}{\lambda_{k,j}}\right)^{(-1)^{\ell}}\right)^{(-1)^m}\\
&=L_{f}(\alpha z)^{(-1)^m}\cdot(-\alpha dz)^{\sum_{\ell=0}^m(-1)^\ell\beta_\ell}\cdot\prod_{k=0}^{m}\prod_{j=1}^{\beta_k} \lambda_{k,j}^{(-1)^{k+1}}\\
&=L_{f}(\alpha z)^{(-1)^m}\,\epsilon (-\alpha dz)^{(-1)^m\chi(M)}.
\end{align*}
Here,
\begin{align*}
\epsilon&=\prod_{k=0}^{m}\prod_{j=1}^{\beta_k} \lambda_{k,j}^{(-1)^{k+1}}=\pm\prod_{k=0}^m\det(f^*_k). \qedhere
\end{align*}
\end{proof}

 We obtain:

\begin{Thm}[{Functional Equation}]\label{FE-case1}
Let $f$ be a continuous map on an {orientable infra-nilmanifold $M=\Pi\bs{G}$} with an affine homotopy lift $(d,D)$. Then the Reidemeister zeta function, whenever it is defined, and the Nielsen zeta function have the following functional equations:
\begin{equation*}
R_{f}\left(\frac{1}{dz}\right)
=\begin{cases}
R_f(z)^{(-1)^m}\epsilon^{(-1)^{p+n}}&\text{when $\Pi=\Pi_+$;}\\
R_f(z)^{(-1)^m}\epsilon^{-1}&\text{when $\Pi\ne\Pi_+$}
\end{cases}
\end{equation*}
and
\begin{equation*}
N_{f}\left(\frac{1}{dz}\right)
=\begin{cases}
N_f(z)^{(-1)^m}\epsilon^{(-1)^{p+n}}&\text{when $\Pi=\Pi_+$;}\\
N_f(z)^{(-1)^m}\epsilon^{-1}&\text{when $\Pi\ne\Pi_+$}
\end{cases}
\end{equation*}
where $d$ is a degree $f$, $m= \dim M$, $\epsilon$ is a constant in $\bbc^\times$, $\sigma=(-1)^n$, $p$ is the number of real eigenvalues of $D_*$ which are $>1$ and $n$ is the number of real eigenvalues of $D_*$ which are $<-1$. If $|d|=1$ then $\epsilon=\pm1$.
\end{Thm}

\begin{proof}
Assume $\Pi=\Pi_+$. Then $R_f(z)= N_f(z)=L_f(\sigma z)^{(-1)^{p+n}}$. By Lemma~\ref{Fried}, we have
\begin{align*}
R_f\left(\frac{1}{dz}\right)=N_f\left(\frac{1}{dz}\right)&= L_f\left(\frac{\sigma}{dz}\right)^{(-1)^{p+n}}\\
&=\left(\epsilon(-\sigma dz)^{(-1)^m\chi(M)}L_f(\sigma z)^{(-1)^m}\right)^{(-1)^{p+n}} \\
&=N_f(z)^{(-1)^m}\epsilon^{(-1)^{p+n}}(-\sigma dz)^{(-1)^{m+p+n}\chi(M)}\\
&= R_f(z)^{(-1)^m}\epsilon^{(-1)^{p+n}}(-\sigma dz)^{(-1)^{m+p+n}\chi(M)}.
\end{align*}

Assume now that $\Pi\ne\Pi_+$. First we claim that $f$ and $f_+$ have the same degree. Let $\pi:M_+\to M$ be the natural double covering projection. Then $\Pi/\Pi_+\cong\bbz_2$ is the group of covering transformations of $\pi$. By \cite[III.2]{Bredon}, the homomorphism $\pi^*:H^m(M;\bbq)\to H^m(M_+;\bbq)$ induces an isomorphism $\pi^*:H^m(M;\bbq)\to H^m(M_+;\bbq)^{\Pi/\Pi_+}$. In particular, $\pi^*$ is injective.
If $x$ is the nontrivial covering transformation, we have the commutative diagram

\centerline {\xymatrix{M_+ \ar[dr]_\pi \ar[rr]^{x}&& M_+\ar[dl]^{\pi}\\
&M&}}
\smallskip

\noindent
This induces the following commutative diagram

\centerline {\xymatrix{H^m(M_+;\bbq)  \ar[rr]^{x^*}&& H^m(M_+;\bbq)\\
&H^m(M;\bbq)\ar[ul]^{\pi^*}\ar[ur]_{\pi^*}&}}
\smallskip

\noindent
We denote generators of $H^m(M;\bbq)$ and $H^m(M_+;\bbq)$ by $[M]$ and $[M_+]$, respectively.
The above diagram shows that $x^*(\pi^*([M]))=\pi^*([M])$, which induces that $x^*([M_+])=[M_+]$ as $\pi^*$ is injective, and hence $x$ acts on $H^m(M_+;\bbq)$ trivially. In other words, $H^m(M_+;\bbq)=H^m(M_+;\bbq)^{\Pi/\Pi_+}$ and $\pi^*:H^m(M;\bbq)\to H^m(M_+;\bbq)$ is an isomorphism. This implies that $f$ and $f_+$ have the same degree.

By Theorem~\ref{T4.5} and Lemma~\ref{Fried}, we have
\begin{align*}
R_f\left(\frac{1}{dz}\right) = N_f\left(\frac{1}{dz}\right) &= L_{f_+}\left(\frac{\sigma}{dz}\right)^{(-1)^{p+n}}\cdot
L_f\left(\frac{\sigma}{dz}\right)^{(-1)^{p+n+1}} \\
&= \left(\epsilon(-\sigma dz)^{(-1)^m\chi(M)}L_{f_+}(\sigma z)^{(-1)^m}\right)^{(-1)^{p+n}}\\
&\quad\x\left(\epsilon(-\sigma dz)^{(-1)^m\chi(M)}L_{f}(\sigma z)^{(-1)^m}\right)^{(-1)^{p+n+1}}\\
&=N_f(z)^{(-1)^m}\left(\epsilon(-\sigma dz)^{(-1)^m\chi(M)}\right)^{-1}\\
&= R_f(z)^{(-1)^m}\left(\epsilon(-\sigma dz)^{(-1)^m\chi(M)}\right)^{-1}.
\end{align*}
On the other hand, it is known that $\chi(M)=0$, e.g. see the remark below, which finishes our proof.
\end{proof}

\begin{Rmk}
Let $G$ be a torsion-free polycylic group. Then $\chi(G)=0$. For, by induction, we may assume that $G$ is an extension of $\bbz^m$ by $\bbz^n$; then as $\chi(\bbz)=0$, we have $\chi(G)=\chi(\bbz^m)\chi(\bbz^n)=0$, \cite[Theorem~6.4.2]{Dekimpe}. {Another proof: A solvmanifold is aspherical and its fundamental group contains a nontrivial Abelian normal subgroup. By Gottlieb's theorem, its Euler characteristic is zero.} If $S$ is a torsion-free extension of $G$ by a finite group of order $k$, then $k\cdot\chi(S)=\chi(G)=0\Rightarrow \chi(S)=0$.
\end{Rmk}

\begin{Rmk}\label{NtoS2}
As it is mentioned above, since Theorem~\ref{T4.5} is true for the Nielsen zeta functions on infra-solvmanifolds of type $\R$, the functional equation for the Nielsen zeta functions in Theorem~\ref{FE-case1} is true on infra-solvmanifolds of type $\R$ (see Theorem~\ref{zeta-S} for the Reidemeister zeta functions).
\end{Rmk}

\section{Asymptotic Nielsen numbers}\label{Asymptotic}

The growth rate of a sequence $a_n$ of complex numbers is defined by
 $$
 \grow (a_n):=\max \left\{1,  \limsup_{n \rightarrow \infty} |a_n|^{1/n}\right\}.
$$
We define  the asymptotic Nielsen  number \cite{I} and the asymptotic Reidemeister number  to be the growth rate
 $N^{\infty}(f) :=  \grow(N(f^n))$ and $ R^{\infty}(f) :=  \grow(R(f^n))$ correspondingly.
 These asymptotic numbers are homotopy type invariants.
 We denote by $\sp(A)$ the spectral radius of the matrix or the operator $A$, $\sp(A)=\lim_n \sqrt[n]{\| A^n \| |}$ which coincide with {the largest  modulus of an eigenvalue of $A$}.
 We denote by $ \bigwedge F_* :=\bigoplus_{\ell=0}^m\bigwedge^\ell F_* $ a linear operator induced in the exterior algebra $\bigwedge^* \bbr^m :=\bigoplus_{\ell=0}^m\bigwedge^\ell \bbr^m $ of $\frakG$ considered as the linear space $\bbr^m$.

\begin{Thm}[{see also the proof of \cite[Theorem~1.5]{mp}}]\label{MP-nil}
Let $M=\Gamma\bs{S}$ be a {special solvmanifold of type $\R$} and let $f$ be a continuous map on $M$ with a Lie group homomorphism $D:S\to S$ as a homotopy lift. Then we have
 $$
 N^{\infty}(f) =\sp(\bigwedge D_{*})
 $$
 provided that $1$ is not in the spectrum of $D_{*}$.
\end{Thm}

\begin{proof}
We give a \emph{very elementary proof} of this theorem. {Compare with the proof of \cite[Theorem~1.5]{mp} in which the authors are interested only in the case of positive topological entropy which excludes the case $N^\infty(f)=1$.}

By \cite[Theorem~2.2]{LL-Nagoya}, we may assume that $f$ is induced by a Lie group homomorphism $D:S\to S$.
Let $\{\lambda_1,\cdots,\lambda_m\}$ be the eigenvalues of $D_*$, counted with multiplicities.
First we note from definition that
\begin{align*}
\sp(\bigwedge D_*)=\begin{cases}\prod_{|\lambda_j|>1}|\lambda_j|&\text{when $\sp(D_*)>1$}\\
1 &\text{when $\sp(D_*)\le1$.}\end{cases}
\end{align*}
In the case when $\sp(D_*)\le1$, the eigenvalues of $\bigwedge^{q\ge1} D_*$ are multiples of eigenvalues of $D_*$, which are $\le1$. On the other hand $\bigwedge^0D_*=\id$, and hence $\sp(\bigwedge D_*)=1$.

Recalling $N(f^n)=|\det(I-D_*^n)|=\prod_{j=1}^m|1-\lambda_j^n|$, we first consider the case where $N(f^n)\ne0$ for all $n$. Then we have
\begin{align*}
\log\limsup_{n\to\infty} N(f^n)^{1/n}&=\limsup_{n\to\infty}\frac{1}{n}\sum_{j=1}^m\log|1-\lambda_j^n|\\
&=\sum_{j=1}^m\limsup_{n\to\infty}\frac{1}{n}\log|1-\lambda_j^n|.
\end{align*}
If $|\lambda|\le1$ then $\limsup_n\frac{1}{n}\log|1-\lambda^n|=0$. For, $\log|1-\lambda^n|\le\log2$. If $|\lambda|>1$ then using L'H\^{o}pital's rule, we have
$$
|\lambda|^n-1\le|1-\lambda^n|\le |\lambda|^n+1
\Rightarrow
\lim_{n\to\infty}\frac{1}{n}\log|1-\lambda^n|=\log|\lambda|.
$$
Hence
\begin{align*}
N^\infty(f)&=\max\left\{1,\limsup_{n\to\infty}N(f^n)^{1/n}\right\}\\
&=\max\left\{1,\prod_{|\lambda|>1}|\lambda|\right\}
=\sp(\bigwedge D_*).
\end{align*}

Next we consider the case where $N(f^n)=0$ for some $n$. Thus some $\lambda_j$ is an $n$th root of unity. For each such $\lambda_j$, consider all $k$'s for which $|1-\lambda_j^k|\ne0$. Since by the assumption $\lambda_j\ne1$, there are infinitely many such $k$'s. Furthermore, there are infinitely many $k$'s for which $|1-\lambda_j^k|\ne0$ for all such (finitely many) $\lambda_j$. Therefore, {when $\sp(D_*)>1$} we have
\begin{align*}
\log\limsup_{n\to\infty} N(f^n)^{1/n}
&=\limsup_{k\to\infty}\frac{1}{k}\log N(f^k)\\
&=\limsup_{k\to\infty}\frac{1}{k}\sum_{j=1}^m\log|1-\lambda_j^k|\\
&=\sum_{|\lambda|>1}\limsup_{k\to\infty}\frac{1}{k}\log|1-\lambda^k|\\
&=\log\left(\prod_{|\lambda|>1}|\lambda|\right);
\end{align*}
{when $\sp(D_*)\le1$ we have $\log\limsup_{n} N(f^n)^{1/n}=0$.}
This completes the proof.
\end{proof}

In fact, what we have shown in the above proof is the following:
\begin{Cor}\label{MP-m}
Let $D$ be a matrix with eigenvalues $\lambda_1,\cdots,\lambda_m$, counted with multiplicities. Let $L(D^n)=\det(I-D^n)$. If $1$ is not in the spectrum of $D$, then
$$
\mathrm{Growth}\left(L(D^n)\right)=\sp(\bigwedge D).
$$
\end{Cor}

\bigskip

Recall that if $f:M\to M$ is a continuous map on an infra-nilmanifold $M=\Pi\bs{G}$ with the holonomy group $\Phi$ and if $f$ has an affine homotopy lift $(d,D)$, then $f$ induces a homomorphism $\phi:\Pi\to\Pi$ defined by the rule:
$$
\forall\alpha\in\Pi,\ \phi(\alpha)\circ(d,D)=(d,D)\circ\alpha.
$$
Furthermore, the homomorphism $\phi$ induces a {function} $\hat\phi:\Phi\to\Phi$ satisfying the identity (\ref{Dekimpe-eq}):
$$
\forall A\in\Phi,\ \hat\phi(A)D=DA.
$$

For any $n\ge1$, we can observe that:
\begin{enumerate}
\item $f^n$ has an affine homotopy lift $(d,D)^n=(*,D^n)$,
\item $f^n$ induces a homomorphism $\phi^n:\Pi\to\Pi$,
\item the homomorphism $\phi^n$ induces a function $\widehat{\phi^n}={\hat\phi}^n:\Phi\to\Phi$.
\end{enumerate}

Recall from Theorem~\ref{AV-all} the averaging formula:
$$
N(f^n)=\frac{1}{|\Phi|}\sum_{A\in\Phi}|\det(I-A_*D_*^n)|.
$$
Since
$$
\frac{1}{|\Phi|}|\det(I-D_*^n)|\le N(f^n),
$$
we have
\begin{align*}
&\frac{1}{n}\log N(f^n)\ge\frac{1}{n}\left(\log|\det(I-D_*^n)|-\log|\Phi|\right)\\
&\Rightarrow
\limsup\frac{1}{n}\log N(f^n)\ge\limsup\frac{1}{n}\left(\log|\det(I-D_*^n)|\right).
\end{align*}
This induces from Corollary~\ref{MP-m} that
\begin{align*}
\sp(\bigwedge D_*)=\mathrm{Growth}(L(D_*^n))\le N^\infty(f).
\end{align*}

Next we recall \cite[Lemma~3.1]{DRM}: Give $A\in\Phi$, we can choose a sequence $(B_i)_{i\in\bbn}$ of elements in $\Phi$ be taking $B_1=A$ and such that $B_{i+1}=\hat\phi(B_i)$, associated to $f$. Since $\Phi$ is finite, this sequence will become periodic from a certain point onwards. Namely, there exist $j,k\ge1$ such that $B_{j+k}=B_j$. It is shown in \cite[Lemma~3.1]{DRM} that
\begin{enumerate}
\item $\forall i\in\bbn,\ \det(I-\hat\phi(B_i)_*D_*)=\det(I-\hat\phi(B_{i+1})_*D_*)$,
\item $\exists \ell\in\bbn$ such that $(\hat\phi(B_j)_*D_*)^\ell=D_*^\ell$,
\end{enumerate}
Since $A$ is of finite order, $\det A_*=\pm1$.

Let $\lambda_1,\cdots,\lambda_m$ be the eigenvalues of $D_*$ counted with multiplicities and let $\mu_1,\cdots,\mu_m$ be the eigenvalues of $\hat\phi(B_j)_*D_*$ counted with multiplicities. Since $(\hat\phi(B_j)_*D_*)^\ell=D_*^\ell$, $(\hat\phi(B_j)_*D_*)^\ell$ has the eigenvalues
$$
\{\lambda_1^\ell, \cdots, \lambda_m^\ell\}=\{\mu_1^\ell,\cdots,\ \mu_m^\ell\}.
$$
We may assume that $\lambda_i^\ell=\mu_i^\ell$ for all $i=1,\cdots,m$. Thus $|\mu_i|=|\lambda_i|$.
Now, \begin{align*}
|\det(I-A_*D_*)|&=|\det A_*\det(A_*^{-1}-D_*)|=|\det(I-D_*A_*)|\\
&=|\det(I-\hat\phi(B_j)_*D_*)|\ \text{ (by (1))}\\
&=\prod_{i=1}^m|1-\mu_i|\le \prod_{i=1}^m(1+|\mu_i|) \ \text{ (by triangle inequality)}\\
&=\prod_{i=1}^m(1+|\lambda_i|)
\end{align*}
Applying the above argument to $D^n$, we obtain that
$$
|\det(I-A_*D_*^n)|\le \prod_{i=1}^m\left(1+|\lambda_i|^n\right).
$$
By the averaging formula, we have
\begin{align*}
N(f^n)&=\frac{1}{|\Phi|}\sum_{A\in\Phi}|\det(I-A_*D_*^n)|\\
&\le\frac{1}{|\Phi|}\sum_{A\in\Phi}\prod_{i=1}^m\left(1+|\lambda_i|^n\right)
=\prod_{i=1}^m\left(1+|\lambda_i|^n\right),
\end{align*}
which induces
\begin{align*}
\limsup\frac{1}{n}\log N(f^n)&\le\sum_{i=1}^m\
\limsup\frac{1}{n}\log\left(1+|\lambda_i|^n\right)\\
&=\sum_{|\lambda|>1} \log|\lambda|=\log\left(\prod_{|\lambda|>1}|\lambda|\right).
\end{align*}
Hence it follows that
\begin{align*}
N^\infty(f)\le\sp(\bigwedge D_*).
\end{align*}

Because the above (algebraic) properties \cite{DRM} and the averaging formula for the Nielsen number \cite{LL-Nagoya} on infra-nilmanifolds can be generalized to infra-solvmanifolds of type $\R$,  we have proven in all that:


\begin{Thm}\label{MP-infranil}
Let $f$ be a continuous map on an infra-solvmanifold of type $\R$ with an affine homotopy lift $(d,D)$. Then we have
$$
N^{\infty}(f) =\sp(\bigwedge D_{*})
$$
provided that $1$ is not in the spectrum of $D_{*}$.
\end{Thm}

\begin{Rmk}
The above theorem was proved when $M$ is a special solvmanifold of type $\R$, see Theorem~\ref{MP-nil} and the proof of \cite[Theorem~1.5]{mp}. {In the paper \cite{mp}, it is assumed that $\sp(D_*)>1$. Since $N(f)=|\det(I-D_*)|$, $1$ is not in the spectrum of $D_*$ if and only if $f$ is not homotopic to a fixed point free map.}
\end{Rmk}

\bigskip

Now, we provide an example of the asymptotic Nielsen numbers.\newline

\begin{Example}
Let $f:\Pi\bs\bbr^2\to\Pi\bs\bbr^2$ be any continuous map on the Klein bottle $\Pi\bs\bbr^2$ of type $(r,\ell,q)$. Recall from \cite[Theorem~2.3]{KLY} and its proof that $r$ is odd or $q=0$, and
\begin{align*}
N(f^n)&=\begin{cases}
|q^n(1-r^n)|&\text{when $r$ is odd and $q\ne0$;}\\
|1-r^n|&\text{when $q=0$,}
\end{cases}\\
D_*&=\begin{cases}
\left[\begin{matrix}r&0\\0&q\end{matrix}\right]&\text{when $r$ is odd and $q\ne0$;}\\
\left[\begin{matrix}r&0\\2\ell&0\end{matrix}\right]&\text{when $q=0$.}
\end{cases}
\end{align*}

Assume $q=0$. If $|r|\le1$, then $N(f^n)\le2$ and so $N^\infty(f)=1$; if $|r|>1$ then
\begin{align*}
\log\limsup_{n\to\infty} N(f^n)^{1/n}&=\limsup_{n\to\infty}\frac{1}{n}\log|1-r^n|
=\log|r|.
\end{align*}
Thus $N^\infty(f)=\max\{1,|r|\}$.

Assume $q\ne0$ and $r$ is odd. If $r=1$ then $N(f^n)=0\Rightarrow N^\infty(f)=1$. If $r\ne1$ is odd,   then
\begin{align*}
\log\limsup_{n\to\infty} N(f^n)^{1/n}&=\limsup_{n\to\infty}\left(\log|q|+\frac{1}{n}\log|1-r^n|\right)\\
&=\begin{cases}
\log|q|&\text{when $|r|\le1$, i.e., $r=-1$;}\\
\log|qr|&\text{when $|r|>1$.}
\end{cases}
\end{align*}
Thus
\begin{align*}
N^\infty(f)=\begin{cases}
1&\text{$q\ne0$ and $r=1$}\\
\max\{1, |q|, |qr|\}&\text{$q\ne0$ and $r\ne1$ is odd}.
\end{cases}
\end{align*}

On the other hand, since $\sp(\bigwedge D_*)$ is the largest modulus of an eigenvalue of $\bigwedge D_*$, it follows that
$$
\sp(\bigwedge D_*)=\max\{1,|r|,|q|,|qr|\}.
$$
Hence:
\begin{enumerate}
\item If $r=1$, then $N^\infty(f)=1$ and $\sp(\bigwedge D_*)=|q|\ge1$ (since $r=1$ is odd and so $q\ne0$).
\item If $r=0$ (even), then $q$ must be $0$ and so $N^\infty(f)=\sp(\bigwedge D_*)=1$.
\item Otherwise, $N^\infty(f)=\sp(\bigwedge D_*)$.
\end{enumerate}

We observe explicitly in this example that the condition that $1$ is not in the spectrum of $D_*$
induce the identity $N^\infty(f)=\sp(\bigwedge D_*)$. {If $q=0$ then $\sp(D_*)=|r|$ and so $N^\infty(f)=\max\{1,|r|\}=\sp(\bigwedge D_*)$.} If $q\ne0$ then $r$ is odd and $\sp(D_*)=\max\{|r|,|q|\}>1$; if $\sp(D_*)=|r|\ge|q|$ then $|r|>1$ and so $N^\infty(f)=|qr|=\sp(\bigwedge D_*)$; if $\sp(D_*)=|q|\ge|r|$ then $|q|>1$ and $|r|>1$ or $r=-1$ (because $r$ cannot be $1$) so $N^\infty(f)=|qr|=\sp(\bigwedge D_*)$.
\end{Example}

\section{Topological entropy and the radius of convergence}\label{EC}

The most widely used measure for the complexity of a dynamical system is the topological
entropy. For the convenience of the reader, we include its definition.
 Let $ f: X \rightarrow X $ be a self-map of a compact metric space. For given $\epsilon > 0 $
 and $ n \in \bbn  $, a subset $E \subset X$ is said to be $(n,\epsilon)$-separated under $f$ if for
 each pair $x \not= y$ in $E$ there is $0 \leq i <n $ such that $ d(f^i(x), f^i(y)) > \epsilon$.
 Let $s_n(\epsilon,f)$  denote the largest cardinality of any $(n,\epsilon)$-separated subset $E$
 under $f$. Thus  $s_n(\epsilon,f)$ is the greatest number of orbit segments ${x,f(x),\cdots,f^{n-1}(x)}$
 of length $n$ that can be distinguished one from another provided we can only distinguish
 between points of $X$ that are  at least $\epsilon$ apart. Now let
 $$
 h(f,\epsilon):= \limsup_{n} \frac{1}{n}\log \,s_n(\epsilon,f)
 $$
 $$
 h(f):=\limsup_{\epsilon \rightarrow 0} h(f,\epsilon).
 $$
 The number $0\leq h(f) \leq \infty $, which to be independent of the metric $d$ used, is called the topological entropy of $f$.
 If $ h(f,\epsilon)> 0$ then, up to resolution $ \epsilon >0$, the number $s_n(\epsilon,f)$ of
 distinguishable orbit segments of length $n$ grows exponentially with $n$. So $h(f)$
 measures the growth rate in $n$ of the number of orbit segments of length $n$
 with arbitrarily fine resolution.

A basic relation between  topological entropy $h(f)$ and Nielsen numbers  was found by N. Ivanov
\cite{I}. We present here a very short proof by B. Jiang  of the Ivanov's  inequality.

\begin{Lemma}[{\cite{I}}]
\label{Iv}
Let $f$ be a continuous map on a compact connected polyhedron $X$. Then
$$
h(f) \geq {\log N^\infty(f)}
$$
\end{Lemma}

\begin{proof}
Let $\delta$ be such that every loop in $X$ of diameter $ < 2\delta $ is contractible.
Let $\epsilon >0$ be a smaller number such that $d(f(x),f(y)) < \delta $ whenever $ d(x,y)<2\epsilon $. Let $E_n \subset X $ be a set consisting of one point from each essential fixed point class of $f^n$. Thus $|E_n| =N(f^n) $. By the definition of $h(f)$, it suffices to show that $E_n$ is $(n,\epsilon)$-separated. Suppose it is not so. Then there would be two points $x\not=y \in E_n$ such that $ d(f^i(x), f^i(y)) \leq \epsilon$ for $o\leq i< n$ hence for all $i\geq 0$. Pick a path $c_i$ from $f^i(x)$ to $f^i(y)$ of diameter $< 2\epsilon$ for $ 0\leq i< n$ and let $c_n=c_0$. By the choice of $\delta$ and $\epsilon$, $f\circ c_i \simeq c_{i+1} $ for all $i$, so $f^n\circ c_0\simeq c_n=c_0$.
This means $x,y$ in the same fixed point class of $f^n$, contradicting the construction of $E_n$.
\end{proof}

This inequality is remarkable in that it does not require smoothness of the map and provides a common lower bound for the topological entropy of all maps in a homotopy class.

Let $ H^*(f): H^*(M;\bbr) \to H^*(M;\bbr) $ be a linear map induced by $f$ on the total cohomology  $H^*(M;\bbr)$ of $M$ with real coefficients. By $\sp(f)$ we denote the spectral radius of $H^*(f)$, which is a homotopy invariant. In 1974 Michael Shub asked, \cite{Shub}, the extent to which the inequality
$$
h(f) \ge \log(\sp(f))
$$
holds. From this time this inequality has been usually called the Entropy Conjecture. Later A. Katok conjectured \cite{Katok} that Entropy Conjecture holds for all continuous map for $M$ being a manifold with the universal cover homeomorphic to $\bbr^m$. In \cite{{mp-a}}, this was confirmed for every continuous map on an infra-nilmanifold.

\begin{Thm}
\label{AB}
Let $f$ be a continuous map on an infra-solvmanifold $M$ of type $\R$  with an affine homotopy lift $(d,D)$.
If $1$ is not in the spectrum of $D_*$, then
$$
h(f)\ge \log(\sp(f)).
$$
If $\bar{f}$ is the map on $M$ induced by the affine map $(d,D)$, then
\begin{align*}
&h(f)\ge h(\bar{f})\ge\log\sp(f),\\
&h(\bar{f})=\log\sp(\bigwedge D_*)=\log N^\infty(\bar{f})=\log N^\infty(f).
\end{align*}
Hence $\bar{f}$ minimizes the entropy in the homotopy class of $f$.
\end{Thm}

\begin{proof}
Let $\bar{f}$ be the map on $M$ induced by the affine map $(d,D)$. Thus $f$ is homotopic to $\bar{f}$. By \cite[Lemma~2.1]{LL-Nagoya}, there is a special solvmanifold which regularly and finitely covers the infra-solvmanifold $M$ so that $f$ can be lifted to $\hat{f}$ on the solvmanifold. We also remark that the Lie group homomorphism $\tau_dD$ induces a map $\bar{\hat{f}}$ on the solvmanifold so that $\bar{f}$ lifts to $\bar{\hat{f}}$, $\hat{f}$ is homotopic to $\phi_f$, the linearization $D_*$ of $f$ is also a linearization of the lift $\hat{f}$, and the topological entropies of $f,\bar{f}$ and their lifts $\hat{f}, \bar{\hat{f}}$ are the same, i.e., $h(f)=h(\hat{f})$ and $h(\bar{f})=h(\bar{\hat{f}})$. Moreover, since the spectral radius is a homotopy invariant, $\sp(f)=\sp(\bar{f})$ and $\sp(\hat{f})=\sp(\bar{\hat{f}})$. It is also known that $\sp(f)\le\sp(\hat{f})$. See, for example, \cite[Proposition~2]{mp-a}.

Now observe that
\begin{align*}
\log\sp(\bigwedge D_*)&=\log N^\infty(f)\ \text{ (Theorem~\ref{MP-infranil})}\\
&=\log N^\infty(\bar{f})\ \text{ (homotopy invariance of $N^\infty(\cdot)$)}\\
&\le h(\bar{f})\ \text{ (Lemma~\ref{Iv})}\\
&=h(\bar{\hat{f}})\ \text{ (lift under a finite regular cover)}\\
&\le\log\sp(\bigwedge D_*).
\end{align*}
{The fundamental is the last inequality. It follows from the estimate of topological entropy of a $C^1$ self-map of a compact manifold $M$
$$
h(f)\le\limsup_{n\to\infty}\frac{1}{n}\log\sup_{x\in M}||\bigwedge Df(x)||
$$
given in \cite{pz} (see \cite{mp} for another reference on this estimate). Next we observe that for an affine map $f=(d,D)$ the latter reduces to $||\bigwedge D||=||\bigwedge D_*||=\sp(\bigwedge D_*)$.}

This implies that
$$
\log\sp(\bigwedge D_*)=\log N^\infty(f)=\log N^\infty(\bar{f})=h(\bar{f}).
$$
Furthermore,
\begin{align*}
\log\sp(\bigwedge D_*)&\ge \log\sp(\bar{\hat{f}})\ \text{ (\cite[Theorem~2.1]{mp})}\\
&=\log\sp(\hat{f})\ \text{ (homotopy invariance of $\sp(\cdot)$)}\\
&\ge\log\sp(f)\ \text{ (lift under a finite regular cover)}.
\end{align*}
Thus we have
\begin{align*}
\log\sp(f) \le \log\sp(\bigwedge D_*)&=\log N^\infty(f)=\log N^\infty(\bar{f})=h(\bar{f})
\\&\le h(f).
\end{align*}
The last inequality follows from Ivanov's inequality, Lemma~\ref{Iv}.
\end{proof}

\begin{Rmk}
If $\sp(D_*)\le1$, then $\sp(\bigwedge D_*)=1$ and so we have the obvious inequality
$$
h(f)\ge0=\log\sp(\bigwedge D_*)\ge\log\sp(f).
$$
\end{Rmk}

\begin{Rmk}
The first inequality $h(f)\ge\log\sp(f)$ in Theorem~\ref{AB}, i.e. Entropy Conjecture, also follows from \cite[Proposition~4.2 and Theorem~1.5]{mp} and by taking a regular finite covering to a special solvmanifold of type $\R$. The second inequality in Theorem~\ref{AB} generalizes the same results \cite[Theorem~4.13]{mp} and \cite[Theorem~B]{mp-a} on nilmanifolds and infra-nilmanifolds.
\end{Rmk}

We denote by $R$ the radius of convergence of the zeta functions $N_f(z)$ or $R_f(z)$.

\begin{Thm}\label{RC}
Let $f$ be a continuous map on an infra-nilmanifold with an affine homotopy lift $(d,D)$. Then the Nielsen zeta function $N_f(z)$ and the Reidemeister zeta function $R_f(z)$, whenever it is defined, have the same positive radius of convergence $R$ which admits following estimation
\begin{equation*}
 R \geq \exp(-h)>0,
\end{equation*}
where $h=\inf \{h(g)\mid g\simeq f\}$.

If $1$ is not in the spectrum of $D_{*}$, the radius $R$ of convergence of $R_f(z)$ is
$$
R=\frac{1}{N^{\infty}(f)}=\frac{1}{\exp h(\bar{f})}
=\frac{1}{\sp(\bigwedge D_{*})}.
$$
\end{Thm}

\begin{proof}
When $R_f(z)$ is defined, as it was observed before, $R(f^n)<\infty$ and so $R(f^n)=N(f^n)>0$ for all $n>0$ on infra-nilmanifolds. In particular, $R_f(z)=N_f(z)$. By the Cauchy-Hadamard formula,
$$
\frac{1}{R}=\limsup_{n\to\infty}\left(\frac{N(f^n)}{n}\right)^{1/n}
=\limsup_{n\to\infty}N(f^n)^{1/n}.
$$
Since $N(f^n)\ge1$ for all $n>0$, it follows that $\limsup_{n\to\infty}N(f^n)^{1/n}\ge1$. Thus
$$
\frac{1}{R}= N^\infty(f)\le\exp h(f).
$$
This induces the inequality $R\ge\exp(-h)$ by the homotopy invariance of the radius $R$ of the Reidemeister zeta function $R_f(z)$. We consider  a smooth map $g:M\rightarrow M$ which is homotopic to $f$. As it is known in \cite{pz}, the entropy $h(g)$ is finite. Thus $\exp(-h) \geq \exp(-h(g)) > 0$. Now the identities in our theorem follow from Theorem~\ref{AB}.

Consider next the Nielsen zeta function $N_f(z)$. If $\limsup_{n\to\infty}N(f^n)^{1/n}\ge1$, then we obtain the same inequality for $R$ as for $R_f(z)$. Thus, we assume $\limsup_{n\to\infty}N(f^n)^{1/n}<1$. This happens only when $N(f^n)=0$ for all but finitely many $n$. In this case, $1/R=\limsup_{n\to\infty}N(f^n)^{1/n}=0$ and so $R=\infty$ and $N^\infty(f)=1$.
\end{proof}

\section{Zeta functions and the Reidemeister torsion of the mapping torus}\label{Tor}

 The Reidemeister torsion
is a graded version of the absolute value of the determinant
of an isomorphism of vector spaces.

Let $d^i:C^i\rightarrow C^{i+1}$ be a cochain complex $C^*$
of finite dimensional vector spaces over $\bbc$ with
$C^i=0$ for $i<0$ and large $i$.
If the cohomology $H^i=0$ for all $i$ we say that
$C^*$ is {\it acyclic}.
If one is given positive densities $\Delta_i$ on $C^i$
then the Reidemeister torsion $\tau(C^*,\Delta_i)\in(0,\infty)$
for acyclic $C^*$ is defined as follows:

\begin{Def}
 Consider a chain contraction $\delta^i:C^i\rightarrow C^{i-1}$,
 i.e., a linear map such that $d\circ\delta + \delta\circ d = \id$.
 Then $d+\delta$ determines a map
 $ (d+\delta)_+ : C^+:=\oplus C^{2i} \rightarrow C^- :=\oplus C^{2i+1}$
 and a map $ (d+\delta)_- : C^- \rightarrow C^+ $.
Since the map $(d+\delta)^2 = id + \delta^2$ is unipotent,
$(d+\delta)_+$ must be an isomorphism.
One defines $\tau(C^*,\Delta_i):= |\det(d+\delta)_+|$.
\end{Def}

Reidemeister torsion is defined in the following geometric setting.
Suppose $K$ is a finite complex and $E$ is a flat, finite dimensional,
complex vector bundle with base $K$.
We recall that a flat vector bundle over $K$ is essentially the
same thing as a representation of $\pi_1(K)$ when $K$ is
connected.
If $p\in K$ is a base point then one may move the fibre at $p$
in a locally constant way around a loop in $K$. This
defines an action of $\pi_1(K)$ on the fibre $E_p$ of $E$
above $p$. We call this action the holonomy representation
$\rho:\pi\to GL(E_p)$.

Conversely, given a representation $\rho:\pi\to GL(V)$
of $\pi$ on a finite dimensional complex vector space $V$,
one may define a bundle $E=E_\rho=(\tilde{K}\times V) / \pi$.
Here $\tilde{K}$ is the universal cover of $K$, and
$\pi$ acts on $\tilde{K}$ by covering transformations and on $V$
by $\rho$.
The holonomy of $E_\rho$ is $\rho$, so the two constructions
give an equivalence of flat bundles and representations of $\pi$.

If $K$ is not connected then it is simpler to work with
flat bundles. One then defines the holonomy as a
representation of the direct sum of $\pi_1$ of the
components of $K$. In this way, the equivalence of
flat bundles and representations is recovered.

Suppose now that one has on each fibre of $E$ a positive density
which is locally constant on $K$.
In terms of $\rho_E$ this assumption just means
$|\det\rho_E|=1$.
Let $V$ denote the fibre of $E$.
Then the cochain complex $C^i(K;E)$ with coefficients in $E$
can be identified with the direct sum of copies
of $V$ associated to each $i$-cell $\sigma$ of $K$.
 The identification is achieved by choosing a basepoint in each
component of $K$ and a basepoint from each $i$-cell.
By choosing a flat density on $E$ we obtain a
preferred density $\Delta_
i$ on $C^i(K,E)$. A case of particular interest is when  $E$ is an acyclic bundle, meaning that the twisted cohomology of $E$ is zero ($H^i(K;E)=0$). In this case one defines the R-torsion of $(K,E)$ to be $\tau(K;E)=\tau(C^*(K;E),\Delta_i)\in(0,\infty)$. It does not depend on the choice of flat density on $E$.

The Reidemeister torsion of an acyclic bundle $E$ on $K$ has many nice properties.
 Suppose that $A$ and $B$ are subcomplexes of $K$. Then we have a multiplicative law:
 \begin{equation}\label{viet}
 \tau(A\cup B;E)\cdot \tau(A\cap B;E) =\tau(A;E)\cdot \tau(B;E)
 \end{equation}
that is interpreted as follows. If three of the bundles $E| A\cup B, \> E |A\cap B, \> E|A,\> E|B$
are acyclic then so is the fourth and the equation (\ref{viet}) holds.

    Another property is the simple homotopy invariance of the Reidemeister torsion.
 In particular $\tau$ is invariant under
subdivision. This implies that for a smooth manifold, one can unambiguously define $\tau(K;E)$ to be the torsion of any smooth triangulation of $K$.

In the case $K= S^1$ is a circle, let $A$ be the holonomy of a generator of the fundamental group
$\pi_1(S^1)$. One has that $E$ is acyclic if and only if $I-A$ is invertible and then
 \begin{equation*}
\tau(S^1;E)= |\det(I-A)|
 \end{equation*}
Note that the choice of generator is irrelevant as $I-A^{-1}= (-A^{-1})(I-A) $ and $|\det(-A^{-1})|=1$.

  These three  properties of the Reidemeister torsion are the analogues of the properties of Euler
  characteristic (cardinality law, homotopy invariance and normalization on a point), but there are
  differences. Since a point has no acyclic representations ($H^0\not=0$) one cannot normalize
  $\tau$  on a point as we do for the Euler characteristic, and so one must use $S^1$ instead.
  The multiplicative cardinality law for the Reidemeister torsion can be made additive just by using $\log\tau$, so the difference here is inessential. More important for some purposes is that
the Reidemeister torsion is not an invariant under a general homotopy equivalence: as mentioned earlier this is in fact why it was first invented.

It might be expected that the Reidemeister torsion counts something geometric (like the Euler characteristic).
D. Fried  \cite{Fri2} showed that it counts the periodic orbits of a flow and the
periodic points of a map.
 We will show that the Reidemeister torsion counts the periodic point classes of a map (fixed point classes of the iterations of the map).

Some further properties of $\tau$ describe its behavior under bundles.
Let  $p: X\rightarrow B$ be a simplicial bundle with fiber $F$ where $F, B, X$ are finite
complexes and $p^{-1}$ sends subcomplexes of $B$ to subcomplexes of $X$
over the circle $S^1$.
We assume here that $E$ is a flat, complex vector bundle over $B$ . We form its pullback $p^*E$
over $X$.
Note that the vector spaces $H^i(p^{-1}(b),\bbc)$ with
$b\in B$ form a flat vector bundle over $B$,
which we denote $H^i F$. The integral lattice in
$H^i(p^{-1}(b),\bbr)$ determines a flat density by the condition
that the covolume of the lattice is $1$.
We suppose that the bundle $E\otimes H^i F$ is acyclic for all
$i$. Under these conditions D. Fried \cite{Fri2} has shown that the bundle
$p^* E$ is acyclic, and
\begin{equation*}
 \tau(X;p^* E) = \prod_i \tau(B;E\otimes H^i F)^{(-1)^i}.
\end{equation*}

Let $f:X\rightarrow X$ be a homeomorphism of
a compact polyhedron $X$.
Let $T_f := (X\times I)/(x,0)\sim(f(x),1)$ be the
mapping torus of $f$.

We shall consider the bundle $p:T_f\rightarrow S^1$
over the circle $S^1$.
We assume here that $E$ is a flat, complex vector bundle with
finite dimensional fibre and base $S^1$. We form its pullback $p^*E$
over $T_f$.
 Note that the vector spaces $H^i(p^{-1}(b),\bbc)$ with
$b\in S^1$ form a flat vector bundle over $S^1$,
which we denote $H^i F$. The integral lattice in
$H^i(p^{-1}(b),\bbr)$ determines a flat density by
the condition
that the covolume of the lattice is $1$.
We suppose that the bundle $E\otimes H^i F$ is acyclic for all
$i$. Under these conditions D. Fried  \cite{Fri2} has shown that the bundle
$p^* E$ is acyclic, and we have
\begin{equation}\label{ReidTor}
 \tau(T_f;p^* E) = \prod_i
 \tau(S^1;E\otimes H^i F)^{(-1)^i}.
\end{equation}
Let $g$ be the preferred generator of the group
$\pi_1 (S^1)$ and let $A=\rho(g)$ where
$\rho:\pi_1 (S^1)\rightarrow GL(V)$.
Then the holonomy around $g$ of the bundle $E\otimes H^i F$
is $A\otimes (f^*)^i$.
Since $\tau(S^1;E)=|\det(I-A)|$ it follows from (\ref{ReidTor})
that
\begin{equation*}
 \tau(T_f;p^* E) = \prod_i \mid\det(I-A\otimes (f^*)^i)\mid^{(-1)^i}.
\end{equation*}
We now consider the special case in which $E$ is one-dimensional,
so $A$ is just a complex scalar $\lambda$ of modulus one.
 Then in terms of the rational function $L_f(z)$ we have :
\begin{equation}\label{ReidTorLef}
 \tau(T_f;p^* E) = \prod_i \mid\det(I-\lambda (f^*)^i)\mid^{(-1)^i}
             = \mid L_f(\lambda)\mid^{-1}
\end{equation}

This means that  the special value of the Lefschetz zeta function is given by the Reidemeister torsion of the corresponding  mapping torus. Let us consider an infra-nilmanifold $M=\Pi\bs{G}$ and a continuous map $f$ on $M$.
As in Section~1, we consider the subgroup $\Pi_+$  of $\Pi$ of index at most $2$. Then $\Pi_+$ is also an almost Bieberbach group as $\Pi$ itself and the corresponding infra-nilmanifold $M_+=\Pi_+\bs{G}$ is a double covering of the infra-nilmanifold $M=\Pi\bs{G}$; the map $f$ lifts to a map $f_+:M_+\to M_+$ which has the same affine homotopy lift $(d,D)$ as $f$. Let $T_f$ and  $T_{f_{+}}$ be the mapping torus  of $f$ and $f_{+}$ correspondingly.
We shall consider two bundles $p:T_f\rightarrow S^1$ and $p_{+}:T_{f_{+}}\rightarrow S^1$ over the circle $S^1$.
We assume here that $E$ is a flat,  complex vector bundle with one dimensional fibre and base $S^1$. We form its pullback $p^*E$ over $T_f$ and pullback $p_{+}^*E$ over $T_{f_{+}}$. We suppose that the bundles $E\otimes H^i M$ and $E\otimes H^i M_{+}$ are acyclic for all $i$.
Then  Theorem~\ref{T4.5} and the formula (\ref{ReidTorLef}) imply
the following result about special values of the Reidemeister and Nielsen zeta functions

\begin{Thm} \label{tor}
Let $f$ be a homeomorphism on an infra-nilmanifold $\Pi\bs{G}$ with an affine homotopy lift $(d,D)$. Then
\begin{align*}
&|R_f((-1)^n\lambda)^{(-1)^{p+n}}| = |R_\phi((-1)^n\lambda)^{(-1)^{p+n}}|
=|N_f((-1)^n\lambda)^{(-1)^{p+n}}|\\
&=\begin{cases}
|L_f(\lambda)|=\tau(T_f;p^* E)^{-1}&\text{when $\Pi=\Pi_+$;}\\
|L_{f_+}(\lambda)L_f(\lambda)^{-1}| =\tau(T_f; p^* E)\tau(T_{f_{+}}; p_{+}^* E)^{-1} &\text{when $\Pi\ne\Pi_+$,}
\end{cases}
\end{align*}
where  $p$ is the number of real eigenvalues of $D_*$ which are $>1$ and $n$ is the number of real eigenvalues of $D_*$ which are $<-1$.
\end{Thm}

\section{Jiang-type spaces and averaging formula for the Reidemeister numbers on infra-solvmanifolds of type $\R$}\label{jiang}

A closed manifold $M$ is called a \emph{Jiang-type space} if for all continuous maps $f:M\to M$,
\begin{align*}
L(f)=0&\Rightarrow N(f)=0;\\
L(f)\ne0&\Rightarrow N(f)=R(f).
\end{align*}
A closed orientable manifold $M$ is called a \emph{Jiang-type space for coincidences} (\cite{GW01}) if for any continuous maps $f,g:N\to M$ where $N$ is any closed orientable manifold of equal dimension,
\begin{align*}
L(f,g)=0&\Rightarrow N(f,g)=0;\\
L(f,g)\ne0&\Rightarrow N(f,g)=R(f,g).
\end{align*}
It is well-known that Jiang spaces are of Jiang-type for coincidences. When $N=M$ is a nilmanifold and $\phi,\psi$ are homomorphisms on the group of covering transformations induced by self-maps $f,g$ on $N$, it is proven in \cite[Theorem~2.3]{Gon} that
\begin{align*}
N(f,g)>0 \Leftrightarrow
\coin(\phi,\psi)=1 \Leftrightarrow
R(f,g)<\infty
\end{align*}
Further if one of the above holds then
$$
R(f,g)=N(f,g)=|L(f,g)|.
$$
Furthermore, nilmanifolds are Jiang-type spaces for coincidences, see \cite{GW01}. Recall that if $N$ is a finite connected complex and $M$ is a nilmanifold then $N(f,g)\ne0\Rightarrow R(f,g)<\infty$; if both $N$ and $M$ are nilmanifolds of equal dimension, then two conditions are equivalent and in that case we have $N(f,g)=R(f,g)$.
\bigskip

Recall what C. McCord proved in \cite[Sec.\!~2]{mccord}.
Let $S_i$ be simply connected solvable Lie groups of type $\E$ with equal dimension, and let $\Gamma_i$ be lattices of $S_i$. Let $D_i:S_1\to S_2$ be Lie group homomorphisms such that $D_i(\Gamma_1)\subset\Gamma_2$. Write $\phi_i=D_i|_{\Gamma_1}:\Gamma_1\to\Gamma_2$. Thus $D_i$ induce maps $f_i$ between orbit spaces $M_i=\Gamma_i\bs{S_i}$, special solvmanifolds of type $\E$. {When $S_i$ are of type $\R$, we can always assume that any $f_i$ is induced from a Lie group homomorphism $D_i$,} {see \cite[Theorem~2.2]{LL-Nagoya} or \cite[Theorem~4.2]{HL-a}.}

Denote $C_\gamma:=\coin(\gamma\circ D_1,D_2)$ and $\bbs_\gamma=p_1\left(\coin(\gamma\circ D_1,D_2)\right)$ for each $\gamma\in\Gamma_2$. We also consider the map $D:S_1\to S_2$ defined by $D(s)=D_1(s)^{-1}D_2(s)$ for $s\in S_1$.

\begin{Lemma}[{\cite[Lemmas~2.6 and ~2.7, and Theorem~2.1]{mccord}}]\label{mcc}
The following are equivalent:
\begin{enumerate}
\item[$(1)$] $\coin(\phi_1,\phi_2)=1$.
\item[$(2)$] $\dim(C_1)=0$.
\item[$(3)$] $D$ is injective.
\item[$(4)$] $C_1=\bbs_1$.
\item[$(5)$] $\ind(\bbs_1)=\pm1$.
\item[$(6)$] $\ind(\bbs_1)\ne0$.
\end{enumerate}
These statements are also valid for any other coincidence class $\bbs_\gamma$, and all $\ind(\bbs_\gamma)$ have the same sign. Hence $N(f_1,f_2)=|L(f_1,f_2)|$.
\end{Lemma}

We generalize \cite[Theorem~2.3]{Gon} from nilmanifolds to special solvmanifolds of type $\R$.

\begin{Thm} \label{Jiang-type}
Let $f_1$ and $f_2$ be maps on a special solvmanifold $\Gamma\bs{S}$ of type $\R$. Let $\phi_1,\phi_2:\Gamma\to\Gamma$ be homomorphisms induced by $f_1,f_2$ respectively. Then the following are equivalent:
\begin{enumerate}
\item[$(\mathrm{a})$] $N(f_1,f_2)>0$.
\item[$(\mathrm{b})$] $\coin(\phi_1,\phi_2)=1$.
\item[$(\mathrm{c})$] $R(f_1,f_2)<\infty$.
\end{enumerate}
Further if one of the above holds then
$$
R(f_1,f_2)=N(f_1,f_2)=|L(f_1,f_2)|.
$$
\end{Thm}

\begin{proof}
By Lemma~\ref{mcc}, $(\mathrm{a})\Leftrightarrow (\mathrm{b})$. Now we will show $(\mathrm{b})\Rightarrow (\mathrm{c})$, and $(\mathrm{c})\Rightarrow (\mathrm{a})$ together with the identity $R(f_1,f_2)=N(f_1,f_2)$.

Let $S$ be a simply connected solvable Lie group of type $\R$. Let $N=[S,S]$ and $\Lambda=S/N$. Then $N$ is nilpotent and $\Lambda\cong\bbr^{k}$ for some $k>0$. A lattice $\Gamma$ of $S$ yields a lattice $N\cap\Gamma$ of $N$. Moreover, the lattice $\Gamma$ induces a short exact sequence $1\to N\cap\Gamma\to\Gamma\to\Gamma/N\cap\Gamma\cong\Gamma\cdot N/N\to1$ so that the following diagram is commutative
$$
\CD
1@>>>N@>>>S@>>>\Lambda=S/N@>>>0\\
@.@AAA@AAA@AAA\\
1@>>> N\cap\Gamma@>>>\Gamma@>>>\Gamma\cdot N/N@>>>0
\endCD
$$
This gives rise to the fibration, called a \emph{Mostow fibration},
$$
N\cap\Gamma\bs{N}\lra M=\Gamma\bs{S}\lra \Gamma\cdot N\bs{S}
$$
over a torus base $\Gamma\cdot{N}\bs{S}$ with compact nilmanifold fiber $N\cap\Gamma\bs{N}$. It is known that this fibration is orientable if and only if the solvmanifold $M$ is a nilmanifold.

Let $E:S\to S$ be a homomorphism. Then $E$ induces a homomorphism $E':N\to N$ and hence a homomorphism $\bar{E}:\Lambda\to\Lambda$ so that the following diagram is commutative
$$
\CD
1@>>>N@>>>S@>>>\Lambda@>>>0\\
@.@VV{E'}V@VV{E}V@VV{\bar{E}}V\\
1@>>>N@>>>S@>>>\Lambda@>>>0
\endCD
$$
Hence we have the following diagram is commutative
$$
\CD
1@>>> N\cap\Gamma@>>>\Gamma@>>>\Gamma\cdot N/N@>>>0\\
@.@VV{\phi'_i}V@VV{\phi_i}V@VV{\bar\phi_i}V\\
1@>>> N\cap\Gamma@>>>\Gamma@>>>\Gamma\cdot N/N@>>>0
\endCD
$$
Denote $\Gamma'=N\cap\Gamma$ and let $\bar\Gamma=\Gamma\cdot N/N$.

By \cite[Theorem~2.2]{LL-Nagoya} or \cite[Theorem~4.2]{HL-a}, we may assume that $f_1,f_2$ are induced by Lie group homomorphisms $D_1,D_2:S\to S$ respectively. Then
$$
\varphi_i(\gamma)\circ D_i=D_i\circ\gamma\ \ \forall\gamma\in\Gamma.
$$
Evaluating at the identity of $S$, we obtain that $\phi_i(\gamma)=D_i(\gamma)$ for all $\gamma\in\Gamma$. So, $\phi_i$ is the restriction of $D_i$ on $\Gamma$.

Assume (b): $\coin(\phi_1,\phi_2)=1$.
Then $\coin(D_1,D_2)=1$ by Lemma~\ref{mcc}. By taking differential, we see that $\coin({D_1}_*, {D_2}_*)=0$, or ${D_2}_*-{D_1}_*$ is a linear isomorphism. We can write ${D_2}_*-{D_1}_*$ as
$$
{D_2}_*-{D_1}_*=\left[\begin{matrix}\bar{D}_{2_*}-\bar{D}_{1_*}&0\\{*}&{D'_2}_*-{D'_1}_*
\end{matrix}\right]
$$
with respect to some linear basis of the Lie algebra of $S$. This implies that $\bar{D}_{2_*}-\bar{D}_{1_*}$ is an isomorphism and so $\coin(\bar{D}_{2_*},\bar{D}_{1_*})=0$ or $\coin(\bar{D}_1,\bar{D}_2)=\bar{1}=\coin(\bar\varphi_1,\bar\varphi_2)$. This happens on $\Lambda\cong\bbr^k$ with the lattice $\Gamma'$ and so on the torus $\Gamma\cdot{N}\bs{S}=\Gamma'\bs\Lambda$. Hence $\coin(\bar\phi_1,\bar\phi_2)=\bar{1}$ implies $R(\bar\phi_1, \bar\phi_2)<\infty$.

On the other hand, since $\coin(\phi'_1, \phi'_2)=1$ from $\coin(\phi_1, \phi_2)=1$, by \cite[Theorem~2.3]{Gon}, $R(\phi_1', \phi_2')<\infty$. Now the above commutative diagram induces a short exact sequence of the sets of Reidemeister classes
$$
\calR(\phi'_1,\phi'_2)\lra \calR(\phi_1,\phi_2)\lra \calR(\bar\phi_1,\bar\phi_2)\lra1.
$$
Because both sets $\calR(\phi'_1,\phi'_2)$ and $\calR(\bar\phi_1,\bar\phi_2)$ are finite, it follows that the middle set $\calR(\phi_1,\phi_2)$ is also finite. Hence $R(\phi_1,\phi_2)<\infty$.

Assume (c): $R(\phi_1,\phi_2)<\infty$. Then $R(\bar\phi_1,\bar\phi_2)<\infty$ on the torus $\Gamma'\bs\Lambda$. We already know that this implies $0<N(\bar{f}_1, \bar{f}_2)=R(\bar\phi_1, \bar\phi_2)$ and $\coin(\bar\phi_1,\bar\phi_2)=\bar{1}$. Assume that $R(\phi'_1,\phi'_2)=\infty$. By \cite[Theorem~2.3]{Gon}, $\coin(\phi'_1,\phi'_2)\ne1$ and then by Lemma~\ref{mcc}, $\coin(D_1',D_2')\ne1$ and hence $D'_{2_*}-D'_{1_*}$ is singular, which implies $D_{2_*}-D_{1_*}$ is also singular and so contradicts $\coin(\phi_1,\phi_2)=1$. Hence $R(\phi'_1,\phi'_2)<\infty$ on the nilmanifold $\Gamma'\bs{N}$. This implies that $0<N(f_1',f_2')=R(\phi'_1, \phi'_2)$. Hence we have
\begin{align*}
N(f_1,f_2)&=|L(f_1,f_2)|\ (\text{\cite[Theorem~2.1]{mccord}})\\
&=|\det(D_{2_*}-D_{1_*})|\ (\text{\cite[Theorem~3.1]{HLP11}})\\
&=|\det(\bar{D}_{2_*}-\bar{D}_{1_*})||\det(D'_{2_*}-D'_{1_*})|\\
&=N(\bar{f}_1,\bar{f}_2)N(f'_1,f'_2)=R(\bar\phi_1,\bar\phi_2)R(\phi'_1,\phi'_2)\\
&\ge R(\phi_1,\phi_2)\ (\text{exactness and finiteness of each Reidemeister set}).
\end{align*}
Consequently, sine it is always true that $N(f_1,f_2)\le R(\phi_1,\phi_2)$, we have the identity $N(f_1,f_2)= R(\phi_1,\phi_2)$.
\end{proof}

Immediately, from Theorem~\ref{Jiang-type} we obtain the following: for any maps $f_1,f_2:M\to M$ on a special solvmanifold $M$ of type $\R$, we have
\begin{align*}
L(f_1,f_2)=0&\Rightarrow N(f_1,f_2)=0;\\
L(f_1,f_2)\ne0&\Rightarrow N(f_1,f_2)=R(f_1,f_2).
\end{align*}

\begin{Example}
Consider the closed $3$-manifolds with $\Sol$-geometry. We refer to \cite[Sec.\!~6]{HL-a} for details about the Reidemeister numbers on these manifolds. These are infra-solvmanifolds $\Pi\bs\Sol$ of type $\R$. When $\Pi=\Pi_0$ or $\Pi_2^\pm$, the corresponding manifold is a torus bundle over $S^1$, and when $\Pi=\Pi_3$ or $\Pi_6$, the manifold is a sapphire space. Only $\Pi_0\bs\Sol$ is the special solvmanifold and the remaining manifolds are non-special, infra-solvmanifolds. For any \emph{homeomorphism} $f:\Pi\bs\Sol\to\Pi\bs\Sol$, let $F_*$ be its linearization. Then the following can be found in \cite[Sec.\!~6]{HL-a}:
\begin{enumerate}
\item When $\Pi=\Pi_0$ or $\Pi_2^+$, $L(f)=N(f)=R(f)=4$ only when $F_*$ is of type (II) with $\det F_*=-1$; otherwise, $L(f)=N(f)=0$ and $R(f)=\infty$.
\item When $\Pi=\Pi_2^-$, $F_*$ is always of type (I) and $L(f)=N(f)=0$, but $R(f)=\infty$.
\item When $\Pi=\Pi_3$, $L(f)=N(f)=0$, but $R(f)=\infty$.
\item When $\Pi=\Pi_6$, $L(f)=N(f)$, which is $0$ or $2$ according as $\det F_*=1$ or $-1$, but $R(f)=\infty$.
\end{enumerate}
These results show that Theorem~\ref{Jiang-type} (i.e., $N(f)>0\Leftrightarrow R(f)<\infty$; in this case, $N(f)=R(f)$) is true for the special solvmanifold $\Pi_0\bs\Sol$ and infra-solvmanifolds $\Pi_2^\pm\bs\Sol$ and $\Pi_3\bs\Sol$, but not true anymore for the infra-solvmanifold $\Pi_6\bs\Sol$.
\end{Example}

Now we can state a practical formula for the Reidemeister number of a pair of continuous maps on an infra-solvmanifold of type $\R$. This is a straightforward generalization of \cite[Theorem~6.11]{HLP} and its proof from infra-nilmanifolds.

\begin{Thm}\label{R-coin}
Let $M=\Pi\bs{S}$ be an infra-solvmanifold of type $\R$ with holonomy group $\Phi$. Let $f,g:M\to M$ be continuous maps with affine homotopy lifts $(d,D), (e,E)$ respectively. Then
$$
R(f,g)=\frac{1}{|\Phi|}\sum_{A\in\Phi}\sigma\left(\det(E_*-A_*D_*)\right),
$$
where $A_*, D_*$ and $E_*$ induced by $A,D$ and $E$ are expressed with respect to a preferred basis of $\Pi\cap S$ and where $\sigma:\bbr\to\bbr\cup\{\infty\}$ is given by $\sigma(0)=\infty$ and $\sigma(x)=|x|$ for all $x\ne0$.
\end{Thm}

\begin{proof}
Choose a fully invariant subgroup $\Lambda\subset\Gamma:=\Pi\cap S$ of $\Pi$ with finite index (\cite[Lemma~2.1]{LL-Nagoya}). Then $f,g$ lift to maps $\bar{f},\bar{g}$ on the special solvmanifold $\Lambda\bs{S}$ of type $\R$ and by \cite[Corollary~1.3]{HL} we have
$$
R(f,g)=\frac{1}{[\Pi:\Lambda]}\sum_{\bar\alpha\in\Pi/\Lambda}R(\bar\alpha\bar{f},\bar{g}).
$$
By Theorem~\ref{Jiang-type}, $R(\bar\alpha\bar{f},\bar{g})=\sigma\!\left(N(\bar\alpha\bar{f},\bar{g})\right)$ for all $\bar\alpha\in\Pi/\Lambda$.

On the other hand, we may assume that $f,g$ are induced by the affine maps $(d,D), (e,E)$ respectively. This induces that $\bar{f}, \bar{g}$ are induced by the Lie group homomorphisms $\mu(d)\circ D, \mu(e)\circ E:S\to S$, where $\mu(\cdot)$ is conjugation. If $(a,A)\in\Pi$ is a preimage of $\bar\alpha\in\Pi/\Lambda$, then the transformation $\bar\alpha$ on $\Lambda\bs{S}$ is induced by the Lie group automorphism $\mu(a)\circ A$. By \cite[Theorem~3.1]{HLP11} and Lemma~\ref{mcc}, we have that
$$
N(\bar\alpha\bar{f},\bar{g})=|\det(\Ad(e)E_*-\Ad(a)A_*\Ad(d)D_*)|
=|\det_\Lambda(E_*-A_*D_*)|
$$
with respect to any preferred basis of $\Lambda$. If we regard this as a basis of $\Gamma$, then we can see that
$$
[\Gamma:\Lambda]\det_\Lambda(E_*-A_*D_*)=\det_\Gamma(E_*-A_*D_*),
$$
for example see the proof of \cite[Theorem~6.11]{HLP}. Hence
\begin{align*}
R(f,g)&=\frac{1}{[\Pi:\Lambda]}\sum_{\bar\alpha\in\Pi/\Lambda}\sigma\!\left(N(\bar\alpha\bar{f},\bar{g})\right)\\
&=\frac{1}{[\Pi:\Lambda]}\sum_{A\in\Phi}[\Gamma:\Lambda]\ \sigma\!\left(\det_\Lambda(E_*-A_*D_*)\right)\\
&=\frac{1}{|\Phi|}\sum_{A\in\Phi}\sigma\left(\det_\Gamma(E_*-A_*D_*)\right).\qedhere
\end{align*}
\end{proof}

{The following corollaries generalize \cite[Theorems~5.1 and 5.2]{DP11} from infra-nilmanifolds to infra-solvmanifolds of type $\R$.}

\begin{Cor}
Let $M=\Pi\bs{S}$ be an orientable infra-solvmanifold of type $\R$.  Let $f,g:M\to M$ be continuous maps. If $R(f,g)<\infty$, then $R(f,g)=N(f,g)$.
\end{Cor}

\begin{proof}
Because $M$ is orientable, the Nielsen number $N(f,g)$ is defined and is equal to, by \cite[Theorem~4.5]{HL-a},
$$
N(f,g)=\frac{1}{|\Phi|}\sum_{A\in\Phi}|\det(E_*-A_*D_*)|.
$$
Since $R(f,g)<\infty$, by Theorem~\ref{R-coin}, $\sigma\!(\det(E_*-A_*D_*))$ is finite for all $A\in\Phi$. By the definition of $\sigma$, we have  $\sigma\!(\det(E_*-A_*D_*))=|\det(E_*-A_*D_*)|$ for all $A\in\Phi$. This finishes the proof.
\end{proof}

\begin{Cor}\label{R-fix}
Let $M=\Pi\bs{S}$ be an infra-solvmanifold of type $\R$ with holonomy group $\Phi$. Let $f:M\to M$ be a continuous map with an affine homotopy lift $(d,D)$. Then
$$
R(f)=\frac{1}{|\Phi|}\sum_{A\in\Phi}\sigma\left(\det(I-A_*D_*)\right),
$$
and if $R(f)<\infty$ then $R(f)=N(f)$.
\end{Cor}


{By Remarks~\ref{NtoS1} and ~\ref{NtoS2}, since the averaging formulas for the Lefschetz number and the Nielsen number are generalized from infra-nilmanifolds to infra-solvmanifolds of type $\R$ (see \cite{HLP11, LL-Nagoya}), all results and proofs concerning the Nielsen number and the Nielsen zeta function in this article directly generalize to the class of infra-solvmanifolds of type $\R$.}

By Corollary~\ref{R-fix} and \cite[Theorem~4.3]{LL-Nagoya}, the averaging formulas for the Reidemeister number and the Nielsen number on infra-solvmanifolds of type $\R$, we can generalize all results and proofs concerning the Reidemeister zeta function, whenever it is defined, to the class of infra-solvmanifolds of type $\R$. If $R_f(z)$ is defined, then $R(f^n)<\infty$ and so by Corollary~\ref{R-fix}  $R(f^n)=N(f^n)>0$ for all $n>0$ and thus $R_f(z)=N_f(z)$. For example, we can  generalize Theorems~\ref{infra}, ~\ref{T4.5}, ~\ref{FE-case1}, and ~\ref{tor}, and their proofs from infra-nilmanifolds to infra-solvmanifolds of type $\R$ to obtain the
following:

\begin{Thm}\label{infrasolv}
Let $f$ be a continuous map on an  infra-solvmanifold of type $\R$ with an affine homotopy lift $(d,D)$.
Assume $N(f)=|L(f)|$. Then the Nielsen zeta function $N_f(z)$ is a rational function and is equal to
\begin{equation*}
N_f(z)=L_f((-1)^qz)^{(-1)^r}
\end{equation*}
where  $q$ is the number of real eigenvalues of $D_*$ which are $<-1$ and
$r$ is the number of real eigenvalues of $D_*$ of modulus $>1$. When the Reidemeister zeta function $R_f(z)$ is defined, we have $R_f(z)=R_\phi(z)=N_f(z)$.
\end{Thm}

\begin{Thm}\label{zeta_infrasolv}
Let $f$ be a continuous map on an infra-solvmanifold $\Pi\bs{S}$ of type $\R$ with an affine homotopy lift $(d,D)$.
Then the Reidemeister zeta function, whenever it is defined, is a rational function and is equal to
\begin{equation*}
R_f(z)=N_f(z)=\begin{cases}
L_f((-1)^nz)^{(-1)^{p+n}}&\text{when $\Pi=\Pi_+$;}\\
\left(\frac{L_{f_+}((-1)^nz)}{L_f((-1)^nz)}\right)^{(-1)^{p+n}}&\text{when $\Pi\ne\Pi_+$,}
\end{cases}
\end{equation*}
\end{Thm}

\begin{Thm}[{Functional Equation}]\label{zeta-S}
Let $f$ be a continuous map on an {orientable infra-solvmanifold $M=\Pi\bs{S}$ of type $\R$} with an affine homotopy lift $(d,D)$. Then the  Reidemeister zeta function, whenever it is defined, and the Nielsen zeta function have the following functional equations:
\begin{equation*}
R_{f}\left(\frac{1}{dz}\right)
=\begin{cases}
R_f(z)^{(-1)^m}\epsilon^{(-1)^{p+n}}&\text{when $\Pi=\Pi_+$;}\\
R_f(z)^{(-1)^m}\epsilon^{-1}&\text{when $\Pi\ne\Pi_+$}
\end{cases}
\end{equation*}
and
\begin{equation*}
N_{f}\left(\frac{1}{dz}\right)
=\begin{cases}
N_f(z)^{(-1)^m}\epsilon^{(-1)^{p+n}}&\text{when $\Pi=\Pi_+$;}\\
N_f(z)^{(-1)^m}\epsilon^{-1}&\text{when $\Pi\ne\Pi_+$}
\end{cases}
\end{equation*}
where $d$ is a degree $f$, $m= \dim M$, $\epsilon$ is a constant in $\bbc^\times$, $\sigma=(-1)^n$, $p$ is the number of real eigenvalues of $D_*$ which are $>1$ and $n$ is the number of real eigenvalues of $D_*$ which are $<-1$. If $|d|=1$ then $\epsilon=\pm1$.
\end{Thm}

\begin{Thm}
Let $f$ be a continuous map on an infra-solvmanifold of type $\R$ with an affine homotopy lift $(d,D)$. Then the Nielsen zeta function $N_f(z)$ and the Reidemeister zeta function $R_f(z)$, whenever it is defined, have the same positive radius of convergence $R$ which admits following estimation
\begin{equation*}
 R \geq \exp(-h)>0,
\end{equation*}
where $h=\inf \{h(g)\mid g\simeq f\}$.

If $1$ is not in the spectrum of $D_{*}$, the radius $R$ of convergence of $R_f(z)$ is
$$
R=\frac{1}{N^{\infty}(f)}=\frac{1}{\exp h(\bar{f})}
=\frac{1}{\sp(\bigwedge D_{*})}.
$$
\end{Thm}

\begin{Thm} \label{tor-S}
Let $f$ be a homeomorphism on an infra-solvmanifold $\Pi\bs{S}$ of type $\R$ with an affine homotopy lift $(d,D)$. Then
\begin{align*}
&|N_f((-1)^n\lambda)^{(-1)^{p+n}}|\\
&=\begin{cases}
|L_f(\lambda)|=\tau(T_f;p^* E)^{-1}&\text{when $\Pi=\Pi_+$;}\\
|L_{f_+}(\lambda)L_f(\lambda)^{-1}| =\tau(T_f; p^* E)\tau(T_{f_{+}}; p_{+}^* E)^{-1} &\text{when $\Pi\ne\Pi_+$,}
\end{cases}
\end{align*}
where  $p$ is the number of real eigenvalues of $D_*$ which are $>1$ and $n$ is the number of real eigenvalues of $D_*$ which are $<-1$.
\end{Thm}

\begin{Rmk}
One may formulate the above theorem also for the Reidemeister zeta function of a homeomorphism $f$ on an infra-solvmanifold of type $\R$ . However it will be seen in Theorem~\ref{Bour} that in the  case of $R_f(z)$ such a manifold must be an infra-nilmanifold.
\end{Rmk}

\begin{Rmk}
{For any map $f$ on an infra-solvmanifold of type $\R$, Theorem~\ref{T4.4} states the relation between the Lefschetz numbers and the Nielsen numbers of iterates of $f$ and Corollary~\ref{R-fix} states the relation of the Nielsen numbers with the Reidemeister numbers of iterates of $f$ when these are finite. Via these relations some of the arithmetic, analytic, and asymptotic properties of the sequences $N(f^n)$ and $R(f^n)$ can be determined from the corresponding properties of the sequence $L(f^n)$. For the sequence $L(f^n)$, all these properties were
thoroughly discussed in \cite[Sect.~\!3.1]{JM}, see also \cite{BaBo}.}
\end{Rmk}

\section{The Reidemeister zeta function is never defined for any homeomorphism of infra-solvmanifold of type $\R$, not an infra-nilmanifold}
\label{No R}

Consider now as an example  closed $3$-manifolds with $\Sol$-geometry. We refer to \cite[Sec.\!~6]{HL-a} for details about the Reidemeister numbers on these manifolds. These are infra-solvmanifolds $\Pi\bs\Sol$ of type $\R$. Let $\Pi_1$ be a lattice of $\Sol$:
$$
\Pi_1=\GammaA=\langle{a_1,a_2,\tau\mid [a_1,a_2]=1,\tau a_i\tau^{-1}=A(a_i)}\rangle,
$$
where $A$ is a $2\x2$-integer matrix of determinant $1$ and trace $>2$. Consider a homomorphism $\phi$ on $\Pi_1$ of type (III), i.e., $\phi$ is given by the form
\begin{align*}
\phi(a_1)=\phi(a_2)=1, \phi(\tau)=a_1^pa_2^q\tau^r, \ r\ne\pm1.
\end{align*}
Then it is shown in \cite[Theorem~6.1]{HL-a} that $R(\phi)=|1-r|$. We can observe easily that $\phi^n$ is also of type (III) and $R(\phi^n)=|1-r^n|$ for all $n>0$. Hence
$$
R_\phi(z)=\exp\left(\sum_{n=1}^\infty\frac{|1-r^n|}{n}z^n\right)
=\begin{cases}
\frac{1}{1-z}&\text{when $r=0$;}\\
\frac{1-\frac{r}{|r|}z}{1-|r|z}&\text{when $|r|>1$.}
\end{cases}
$$
It can be seen also that if $\phi$ is not of type (III), then $R(\phi)=\infty$ or $R(\phi^2)=\infty$. Thus the associated Reidemeister zeta function is not defined.
A similar phenomenon happens for the infra-solvmanifold $\Pi^\pm\bs\Sol$. For the remaining infra-solvmanifolds $\Pi_3\bs\Sol$ and $\Pi_6\bs\Sol$, it is shown that only trivial map has a finite Reidemeister number, which is $1$. That is, only the trivial map defines the Reidemeister zeta function. {The homomorphisms above  are eventually commutative, and in fact,  for every  eventually commutative homomorphism  the
Reidemeister zeta function, whenever it is defined, is a rational function, see Theorem 9 and Theorem 10 in \cite{Fel00}.}

We will show now that if  the Reidemeister zeta function is defined for a homeomorphism on an infra-solvmanifold of type $\R$, then the manifold must be an infra-nilmanifold.

Recall the following

\begin{Prop}[{\cite[Ex.\,21(b), p.\,97]{Bourbaki}, \cite[Proposition~3.6]{Smale}}]
\label{BS}
Let $\sigma$ be a Lie algebra automorphism. If none of the eigenvalues of $\sigma$ is a root of unity, then the Lie algebra must be nilpotent.
\end{Prop}

\begin{Thm}\label{Bour}
If the Reidemeister zeta function $R_f(z)$ is defined for a homeomorphism $f$ on an infra-solvmanifold $M$ of type $\R$, then $M$ is an infra-nilmanifold.
\end{Thm}

\begin{proof}
Let $f$ be a {homeomorphism} on an infra-solvmanifold $M=\Pi\bs{S}$ of type $\R$. By \cite[Theorem~2.2]{LL-Nagoya}, we may assume that $f$ has an affine map as a homotopy lift. By \cite[Lemma~2.1]{LL-Nagoya}, there is a special solvmanifold $N=\Lambda\bs{S}$ which covers $M$ finitely and on which $f$ has a lift $\bar{f}$, which is induced by a Lie group automorphism $D$ on the solvable Lie group $S$.

From \cite[Corollary~1.3]{HL}, we have an averaging formula for Reidemeister numbers:
$$
R(f^n)=\frac{1}{[\Pi:\Lambda]}\sum_{\bar\alpha\in\Pi/\Lambda}R(\bar\alpha\bar{f}^n).
$$

Assume now that $f$ defines the Reidemeister zeta function. Then $R(f^n)<\infty$ for all $n>0$. The above averaging formula implies  that $R(\bar{f}^n)<\infty$ for all $n$. By Theorem \ref{Jiang-type}, we must have
$$
R(\bar{f}^n)=N(\bar{f}^n)=|L(\bar{f}^n)|>0.
$$
Since $L(\bar{f}^n)=\det(I-D_*^n)\ne0$ for all $n>0$ by \cite[Theorem~3.1]{HLP11}, this would imply that the differential $D_*$ of $D$ has no roots of unity. By Proposition~\ref{BS}, $S$ must be nilpotent.
\end{proof}

\begin{Rmk} Let $A$ be an Anosov diffeomorphism on an infra-nilmanifold.  Then an iteration  $A^n$ will be also an Anosov diffeomorphism for every $n \geq 1$. The Reidemeister number of an Anosov diffeomorphism
is always finite \cite{DRPAn}. Hence the Reidemeister zeta  $R_A(z)$  is well-defined.
From Theorem~\ref{T4.5} and Theorem~\ref{FE-case1} it follows that the Reidemeister zeta function $R_A(z)$ of an Anosov diffeomorphism on an infra-nilmanifold is a rational function with functional equation.
It is known that a nilmanifold modelled on a free c-step nilpotent Lie group on $r$ generators admits an Anosov diffeomorphism if and only if
$r > c$ \cite{Da}. Hence the Reidemeister zeta function  of an Anosov diffeomorphism on such nilmanifold is well-defined  if $r > c$ and is a rational function with functional equation.
\end{Rmk}

\section{The Artin-Mazur zeta functions on infra-solvmanifolds of type $\R$} \label{AM}

Let $f$ be a continuous map on a topological space $X$. Then the \emph{Artin-Mazur zeta function} of $f$ is defined as follows:
$$
AM_f(z)=\exp\left(\sum_{n=1}^\infty\frac{F(f^n)}{n}z^n\right)
$$
where $F(f)$ is the number of isolated fixed points of $f$.

\begin{Prop}[{\cite[Proposition~1]{KL}}]\label{KL}
Let $f$ be a continuous map on an infra-solvmanifold $\Pi\bs{S}$ of type $\E$ induced by an affine map $F:S\to S$. For any $\alpha\in\Pi$, $\fix(\alpha\circ F)$ is an empty set or path connected. Hence every nonempty fixed point class of $f$ is path connected, and every isolated fixed point class forms an essential fixed point class.
\end{Prop}

\begin{proof}
Let $x,y\in\fix(\alpha\circ F)$. So, the affine map $\alpha F$ fixes $x$ and $y$. Writing $\alpha\circ F=(d,D)\in S\rtimes\Endo(S)$, we see that
\begin{itemize}
\item $(d,D)({x})={x}\Rightarrow D({x})=d^{-1}{x}$,
\item $(d,D)({y})={y}\Rightarrow D({y})=d^{-1}{y}$,
\item $({x},I)^{-1}(\alpha\circ F)({x},I)=({x},I)^{-1}(d,D)({x},I)=({x}^{-1}dD({x}),D)=(1,D)$ and $D$ fixes $1$ and ${x}^{-1}{y}$.
\end{itemize}
Since $S$ is of type $\E$, $\exp:\frakS\to S$ is a diffeomorphism with inverse $\log$. Let $X=\log({x}^{-1}{y})\in\frakS$. Then the $1$-parameter subgroup $\{\exp(tX)\mid t\in\bbr\}$ of $S$ is fixed by the endomorphism $D$. Consequently, the affine map $\alpha\circ F$ fixes the `line' connecting the points ${x}$ and ${y}$. In particular, $p(\fix(\alpha\circ F))$ is isolated $\{\bar{x}\}$ if and only if $\fix(\alpha\circ F)$ is isolated $\{x\}$, where $p:S\to\Pi\bs{S}$ is the covering projection. Further, the index of the fixed point class $p(\fix(\alpha\circ F))=\{\bar{x}\}$ is
{
$$
\det(I-df_{\bar{x}})=\pm\det(I-d(\alpha\circ F)_x)=\pm\det(I-D_*)
$$
}
where the second identity follows from the fact that $x^{-1}(\alpha\circ F)x=D$. Since $D$ fixes only the identity element of $S$, $D_*$ fixes only the zero element of $\frakS$ and so $I-D_*$ is nonsingular. Hence the fixed point class $p(\fix(\alpha\circ F))$ is essential.
\end{proof}

\begin{Rmk}
The above proposition is a straightforward generalization of \cite[Proposition~1]{KL} from infra-nilmanifolds to infra-solvmanifolds of type $\E$. Further, the linear part of the affine map $F$ need not be an automorphism.
\end{Rmk}

Proposition~\ref{KL} is proved when the manifold is a special solvmanifold of type $\E$ and the map is induced by a homomorphism in Lemma~\ref{mcc}, \cite{mccord}. In fact, the converse is also proved. That is, every essential fixed point class consists of a single element. We will prove the converse of the proposition on infra-solvmanifolds of type $\R$.

\begin{Prop}\label{single}
Let $f$ be a continuous map on an infra-solvmanifold of type $\R$ induced by an affine map. Then every essential fixed point class of $f$ consists of a single element.
\end{Prop}

\begin{proof}
Let $\tilde{f}=(d,D)$ be the affine map on the connected, simply connected solvable Lie group $S$ of type $\R$ which induces $f:\Pi\bs{S}\to\Pi\bs{S}$. Then $f$ induces a homomorphism $\phi:\Pi\to\Pi$.

By \cite[Lemma~2.1]{LL-Nagoya}, we can choose a fully invariant subgroup $\Lambda\subset\Pi\cap S$ of $\Pi$ with finite index. Hence $\phi(\Lambda)\subset\Lambda$. This implies that $\tilde{f}$ induces a map $\bar{f}$ on $\Lambda\bs{S}$.

Then we have an averaging formula, \cite[Theorem~4.2]{LL-Nagoya},
$$
N(f)=\frac{1}{[\Pi:\Lambda]}\sum_{\bar\alpha\in\Pi/\Lambda}N(\bar\alpha\circ\tilde{f}).
$$

Assume that $f$ has an essential fixed point class. The averaging formula tells that this essential fixed point class of $f$ is lifted to an essential fixed point class of some $\bar\alpha\circ\bar{f}$. That is, there is $\alpha=(a,A)\in\Pi$ such that the fixed point class $p'(\fix(\alpha\circ\tilde{f}))$ of $\bar\alpha\circ\bar{f}$ is essential (and so $p(\fix(\alpha\circ\tilde{f}))$ is an essential fixed point class of $f$). It suffices to show that the fixed point class $p'(\fix(\alpha\circ\tilde{f}))$ consists of only one point.

Let $F=\alpha\circ\tilde{f}=(a,A)(d,D):=(e,E)$ be the affine map on $S$, and let $\bar{F}=\bar\alpha\circ\bar{f}$. Then $p'(\fix(F))$ is essential and $N(\bar{F})=|\det(I-E_*)|\ne0$. Choose $x\in\fix(F)=\fix((e,E))$. Then the left multiplication by $x^{-1}$,
$$
\ell_{x^{-1}}:y\in\fix((e,E))\mapsto x^{-1}y\in\fix(E),
$$
is a bijection. Further, since $\exp:\frakS\to S$ is a diffeomorphism, it follows that $\fix(E) \leftrightarrow\ \gfix(E_*)=\ker(I-E_*)$. Since $I-E_*$ is invertible, we see that $\fix(F)$ and hence $p'(\fix(F))$ and $p(\fix(F))$ consist of a single element.
\end{proof}

\begin{Rmk}\label{isolated}
In Propositions~\ref{KL} and \ref{single}, we have shown that for any continuous map on an infra-solvmanifold of type $\R$ induced by an affine map the isolated fixed points of $f$ are the essential fixed point classes of $f$. That is, $F(f)=N(f)$. Similarly  $F(f^n)=N(f^n)$ for all $n$.
\end{Rmk}
Therefore, by Theorem \ref{zeta_infrasolv} and Theorem~\ref{zeta-S} we have

\begin{Thm}
Let $f$ be a continuous map on an infra-solvmanifold of type $\R$ induced by an affine map. Then
$
AM_f(z)=N_f(z),
$
i.e.,  $AM_f(z)$ is a rational function with functional equation.
\end{Thm}

By the main result in \cite{mccord94}, if $f$ is a map on an infra-solvmanifold of type $\R$ which is induced by an affine map and is  homotopically periodic, then we have $AM_f(z)=N_f(z)=L_f(z)$ as $N(f^n)=L(f^n)$.
According to Theorem~\ref{virtual unipotency}, if $f$ is a virtually unipotent affine diffeomorphism on an infra-solvmanifold of type $\R$, then we still have $AM_f(z)=N_f(z)=L_f(z)$.

\section{The Nielsen numbers of virtually unipotent maps on infra-solvmanifolds of type $\R$}

A square matrix is \emph{unipotent} if all of its eigenvalues are $1$. A square matrix is called \emph{virtually unipotent} if some power of it is unipotent.

Let $M=\Pi\bs{S}$ be an infra-solvmanifold of type $\R$. Let $f:M\to M$ be a continuous map with an affine homotopy lift $(d,D)\in\aff(S)$. Then $f$ is homotopic to the diffeomorphism on $M$ induced by the affine map $(d,D)$, called an \emph{affine diffeomorphism}. If, in addition, $D_*$ is virtually unipotent then we say that $f$ is a \emph{virtually unipotent} map.

Now we observe the following:
\begin{enumerate}
\item A matrix is virtually unipotent if and only if all of its eigenvalues have absolute value $1$, see \cite[Lemma~11.6]{ST}.
\item Let $\Phi$ be a finite subgroup of $\GL(n,\bbr)$ and let $D \in \GL(n,\bbr)$ normalize $\Phi$. If $D$ is virtually unipotent, then for all $A \in \Phi$, $AD$ is virtually unipotent, see \cite[Lemma~3.2]{Malfait}.
\end{enumerate}

\begin{Example}
Consider the $3$-dimensional Lie group $\Sol=\bbr^2\rtimes_\sigma\bbr$, where
\begin{align*}
\sigma(t)=\left[\begin{matrix}e^t&0\\0&e^{-t}\end{matrix}\right].
\end{align*}
Let $g=((x,y),t)\in\Sol$. Then it can be seen easily that $\tau_g:\Sol\to\Sol$ is given by
$$
\tau_g:((u,v),s)\mapsto(e^tu-e^sx+x,e^{-t}v-e^{-s}y+y),s),
$$
and $\Ad(g):\sol\to\sol$ is given by
$$
\Ad(g)=\left[\begin{matrix}e^t&0&-x\\0&e^{-t}&\hspace{8pt}y\\0&0&\hspace{8pt}1\end{matrix}\right]
$$
for some basis of $\sol$. Hence $\Ad(g)$ is not virtually unipotent unless $t=0$.

Now consider the infra-solvmanifold $\Pi_2^+\bs\Sol$.
The holonomy group of $\Pi_2^+\bs\Sol$ is
$$
\Phi_2^+=\left\langle
\left[\begin{matrix}-1&\hspace{8pt}0&0\\\hspace{8pt}0&-1&0\\\hspace{8pt}0&\hspace{8pt}0&1\end{matrix}\right]
\right\rangle
$$
and thus $\Sol^\Phi=\{((x,y),t)\in\Sol\mid x=y=0\}$. Fix $g=((0,0),t)\in\Sol^\Phi$ with $t\ne0$ and consider $(g,\tau_{g^{-1}})\in\aff(\Sol)$. Then $(g,\tau_{g^{-1}})$ centralizes $\Pi_2^+$ and $(g,\tau_{g^{-1}})$ induces an affine diffeomorphism $f$ on $\Pi_2^+\bs\Sol$ given by $\bar{x}\mapsto \bar{x}\bar{g}$. Hence the affine diffeomorphism $f$ is homotopic to the identity map. However $f$ is not virtually unipotent since $(\tau_{g^{-1}})_*=\Ad(g^{-1})$ is not virtually unipotent.
\end{Example}

\begin{Rmk}
Recall \cite[Lemma~3.6]{Malfait}, which states that if an affine diffeomorphism $f$ on an infra-nilmanifold $M$ is homotopic to a virtually unipotent affine diffeomorphism on $M$, then $f$ is virtually unipotent. However, the above example shows that this statement is not true in general for infra-solvmanifolds of type $\R$. Namely, there is an affine diffeomorphism on an infra-solvmanifold of type $\R$ which not virtually unipotent, but is homotopic to a virtually unipotent affine diffeomorphism.

Furthermore, in the above example, $f\simeq\id$ is a homotopically periodic map which is not virtually unipotent. Therefore \cite[Proposition~3.11]{Malfait} is not true in general for infra-solvmanifolds of type $\R$. Note also that there is a unipotent affine diffeomorphism on the torus which is not homotopically periodic, see \cite[Remark~3.12]{Malfait}.

Consequently, on infra-nilmanifolds homotopically periodic maps are virtually unipotent maps. But on infra-solvmanifolds of type $\R$, there is no relation between homotopically periodic maps and virtually unipotent maps.
\end{Rmk}

\begin{Thm}\label{virtual unipotency}
If $f$ is a virtually unipotent map on an infra-solvmanifold of type $\R$, then $L(f)=N(f)$.
\end{Thm}

\begin{proof}
Let $M$ be an infra-solvmanifold of type $\R$ with holonomy group $\Phi$. Then we can assume $f$ is an affine diffeomorphism induced by an affine map $(d,D)$ such that $D_*$ is virtually unipotent. This implies that $(d,D)$ normalizes $\Pi$ and hence it follows that $D$ normalizes the holonomy group $\Phi$. By the previous observation (2), since $D_*$ is virtually unipotent, so are all $A_*D_*$ where $A\in\Phi$ and hence by \cite[Lemma~4.2]{Malfait}, $\det(I-A_*D_*)\ge0$. Using the averaging formula \cite[Theorem~4.3]{LL-Nagoya}, we obtain
\begin{align*}
N(f)&=\frac{1}{|\Phi|}\sum_{A\in\Phi}|\det(I-A_*D_*)|\\
&=\frac{1}{|\Phi|}\sum_{A\in\Phi}\det(I-A_*D_*)=L(f).\qedhere
\end{align*}
\end{proof}

\section{Gauss congruences for the Nielsen and Reidemeister numbers}\label{Gauss cong}

In number theory, the following Gauss congruence for integers holds:
$$
\sum_{d\mid n}\mu(d)\ a^{n/d}\equiv 0\mod n
$$
for any integer $a$ and any natural number $n$. Here $\mu$ is the M\"obius function. In the case of a prime power $n=p^r$, the Gauss congruence turns into the Euler congruence. Indeed, for $n=p^r$ the M\"obius function $\mu(n/d)=\mu(p^r/d)$ is different from zero only in two cases: when $d=p^r$ and when $d=p^{r-1}$. Therefore, from the Gauss congruence we obtain the Euler congruence
$$
a^{p^r}\equiv a^{p^{r-1}}\mod p^r
$$
This congruence is equivalent to the following classical Euler's theorem:
$$
a^{\varphi(n)}\equiv 1\mod n
$$
where $(a,n)=1$.

These congruences have been generalized from integers $a$ to some other mathematical invariants such as the traces of all integer matrices $A$ and the Lefschetz numbers of iterates of a map, {see \cite{mp99,Z}}:
\begin{align}
\label{Gauss}
&\sum_{d\mid n}\mu(d)\ \tr(A^{n/d})\equiv 0\mod n,\\
\label{Euler}
&\tr(A^{p^r})\equiv\tr(A^{p^{r-1}})\mod p^r.
\end{align}
{A. Dold in \cite{Dold} proved by a geometric argument the following congruence for the fixed point index of iterates of a map $f$ on a compact ANR $X$ and any natural number $n$
$$
\sum_{d\mid n}\mu(d)\ \ind(f^{n/d},X)\equiv 0\mod n,
$$
thus consequently
\begin{align}
\label{Dold}
\sum_{d\mid n}\mu(d)\ L(f^{n/d})\equiv 0\mod n\tag{DL}
\end{align}
by using Hopf theorem. These congruences are now called the Dold congruences.} It is also shown in \cite{mp99} (see also \cite[Theorem~9]{Z}) that the above congruences \eqref{Gauss}, \eqref{Euler} and \eqref{Dold} are equivalent. For example, $\eqref{Gauss}\Rightarrow \eqref{Dold}$ follows easily by the following observation: Let $A_i$ be an integer matrix obtained from the homomorphism $f_{i_*}:H_i(X;\bbq)\to H_i(X;\bbq)$. Then
\begin{align*}
\sum_{d\mid n}\mu\!\left(\frac{n}{d}\right) L(f^{d})
&=\sum_{d\mid n}\mu\!\left(\frac{n}{d}\right)\left(\sum_i(-1)^i \tr(A_i^d)\right)\\
&=\sum_i(-1)^i\left(\sum_{d\mid n}\mu\!\left(\frac{n}{d}\right)\tr(A_i^d)\right)\\
&\equiv\sum_i(-1)^i\ 0=0\mod n.
\end{align*}
Moreover, we have
\begin{align}\label{EL}
L(f^{p^r})&=\sum_i(-1)^i \tr(A_i^{p^r})
\equiv\sum_i(-1)^i\tr(A_i^{p^{r-1}})\tag{EL}\\
&=L(f^{p^{r-1}})\mod p^r.\notag
\end{align}

Now we shall consider the following congruences for the Nielsen numbers and Reidemeister numbers
\begin{align}
\label{DR}
&\sum_{d\mid n}\mu(d)\ R(f^{n/d})\equiv 0\mod n,\tag{DR}\\
\label{ER}
&R(f^{p^r})\equiv R(f^{p^{r-1}})\mod p^r,\tag{ER}\\
\label{DN}
&\sum_{d\mid n}\mu(d)\ N(f^{n/d})\equiv 0\mod n,\tag{DN}\\
\label{EN}
&N(f^{p^r})\equiv N(f^{p^{r-1}})\mod p^r\tag{EN}
\end{align}
and find the relations between them and the conditions on spaces, groups and/or on maps for which the congruences hold true.

\begin{Example}
Let $f$ be a map on an infra-solvmanifold of type $\R$ which is homotopically periodic or virtually unipotent. Then $N(f^n)=L(f^n)$ for all $n>0$.  The congruence \eqref{Dold} immediately implies the congruences \eqref{DN} and \eqref{EN}.
\end{Example}

\begin{Example}
Let $f:S^2\vee S^4\rightarrow S^2\vee S^4$ be the map considered in Example~\ref{wedge}. Then
\begin{align*}
&L(f)=N(f)=0,\ L(f^k)=2+(-2)^k,\  N(f^k)=1\ \ \forall k>1,\\
&R(f^k)=1\ \ \forall k\geq 1.
\end{align*}
Thus we have no congruence \eqref{DN} and we have nice congruences \eqref{DR} and \eqref{Dold}.
\end{Example}

\begin{Example}
Let $f$ be a map on the circle $S^1$ of degree $d$. Then $N(f^n)=|L(f^n)|=|1-d^n|(=R(f^n)$ when $d\ne\pm1$). When $d\ne\pm1$, then all $R(f^n)<\infty$ and so the congruences hold. When $d=1$, the congruence for the Nielsen number is obviously true. We assume $d=-1$. So, $N(f^n)=2$ for odd $n$ and $0$ for even $n$. For $n=2\cdot3^2\cdot5$, we have
\begin{align*}
\sum_{d\mid n}\mu(d)\ N(f^{n/d})&=\sum_{\substack{d\mid n\\ d\text{ even}}}\mu(d)\ 2\\
&=2\left(\mu(2)+\mu(2\cdot3)+\mu(2\cdot3^2)+\mu(2\cdot3\cdot5)+\mu(2\cdot3^2\cdot5)\right)\\
&=2\left((-1)+1+0+(-1)+0\right)=-2\\
&\ne0\mod 2\cdot3^2\cdot5.
\end{align*}
Thus we have no congruence \eqref{DN}.

Next we consider the congruences \eqref{EN} and \eqref{ER}.
If $d\ge0$, then $L(f^n)=1-d^n=-N(f^n)=-R(f^n)$. The congruence \eqref{EL} $L(f^{p^r})\equiv L(f^{p^{r-1}})\mod p^r$ implies the other congruences \eqref{EN} and \eqref{ER}. Assume $d<0$. The congruence \eqref{EL} $L(f^{p^r})\equiv L(f^{p^{r-1}})\mod p^r$ is exactly $1-d^{p^r}\equiv 1-d^{p^{r-1}}\mod p^r$, which implies that $d^{p^r}\equiv d^{p^{r-1}}\mod p^r$ and so $d^{p^r}\pm1\equiv d^{p^{r-1}}\pm1\mod p^r$. Thus the other congruences \eqref{EN} and \eqref{ER} hold true.

In summary, \eqref{EN} and \eqref{ER} are true, but \eqref{DN} is not true.
\end{Example}

The congruence (\ref{DR}) was previously known for automorphisms of almost polycyclic groups (\cite[p.~\!195]{crelle}) and for all continuous maps only on nilmanifolds (\cite[Theorem~58]{Fel00}) provided that all Reidemeister numbers of iterates of the maps are finite. We generalize these on infra-solvmanifolds of type $\R$.

\begin{Thm}\label{congruence}
Let $f$ be any continuous map on an infra-solvmanifold of type $\R$ such that all $R(f^n)$ are finite. Then we have
$$
\sum_{d\mid n}\mu(d)\ R(f^{n/d})=\sum_{d\mid n}\mu(d)\ N(f^{n/d})\equiv0\mod n
$$
for all $n>0$.
\end{Thm}

\begin{proof}
We define
\begin{align*}
P^n(f)&=\text{the set of isolated periodic points of $f$ with period $n$},\\
P_d(f)&=P^d(f)-\bigcup_{k\mid d}P^k(f)\\
&=\text{ the set of isolated periodic points of $f$ with least period $d$}.
\end{align*}
Then we have
$$
P^n(f)=\coprod_{d\mid n}P_d(f)\ \text{ or } \#P^n(f)=\sum_{d\mid n}\#P_d(f).
$$
By the M\"obius inversion formula when all terms are finite, we have
$$
\#P_n(f)=\sum_{d\mid n} \mu(d)\ \#P^{n/d}(f).
$$
On the other hand, if $x\in P_n(f)$ then $f(x)\in P_n(f)$. For, $f^k(f(x))=f(x)\Rightarrow f^{n-1}(f^k(f(x))=f^{n-1}(f(x))\Rightarrow f^k(x)=x$, showing that $x$ and $f(x)$ have the same least period $n$. It remains to show that if $x$ is isolated, then $f(x)$ is isolated. Let $U$ be a neighborhood of $x$ containing no other periodic points of period $n$. Then the inverse image $V$ of $U$ under $f^{n-1}$ is a neighborhood of $f(x)$. If $y\in V$ is a periodic point of $f$ with period $n$, then $f^{n-1}(y)\in U$ and so $f^{n-1}(y)=x \Rightarrow y=f^n(y)=f(x)$, which shows that $f(x)$ is isolated. Thus $f$ maps $P_n(f)$ into $P_n(f)$ and this implies that $P_n(f)$ is the disjoint union of $f$-orbits, each of length $n$. So, when $\#P_n(f)$ is finite, it is a multiple of $n$.

Let $M$ be an infra-solvmanifold of type $\R$ and let $f$ be a map on $M$. Since we are working with the Nielsen numbers and the Reidemeister numbers of iterates of $f$, we may assume that $f$ is induced by an affine map on $S$.

Assume $R(f^n)<\infty$; then $N(f^n)=R(f^n)>0$ by Corollary~\ref{R-fix}. By Remark~\ref{isolated}, $N(f^n)$ is the number of isolated periodic points of $f$ with period $n$; $N(f^n)=\#P^n(f)$.

Consequently, what we have shown is that if all $R(f^{n})<\infty$, then
\begin{align*}
\sum_{d\mid n}\mu(d)\ R(f^{n/d})=\sum_{d\mid n}\mu(d)\ N(f^{n/d})=\#P_n(f)\equiv0\mod n.
\end{align*}
This proves our theorem.
\end{proof}

\begin{Cor}\label{cor:congruence}
Let $f$ be a map on an infra-solvmanifold of type $\R$ which is homotopic to an affine diffeomorphism induced by an affine map $(d,D)$. If $D_*$ has no eigenvalue that is a root of unity, then all $R(f^n)$ are finite. Hence the Gauss congruences \eqref{DR} for the Reidemeister and \eqref{DN} for the Nielsen numbers hold true.
\end{Cor}

\begin{proof}
This follows from a straightforward generalization of \cite[Proposition~4.3]{DRPAn} from infra-nilmanifolds to infra-solvmanifolds of type $\R$.

Let $M=\Pi\bs{S}$ be the infra-solvmanifold of type $\R$ with holonomy group $\Phi$. Recall that $f$ induces a homomorphism $\phi:\Pi\to\Pi$ given by $\phi(\alpha)\circ(d,D)=(d,D)\circ\alpha$ for all $\alpha\in\Pi$. That is, $\phi=\tau_{(d,D)}$. This implies that $(d,D)$ normalizes $\Pi$ and hence $D$ normalizes $\Phi$. So $AD^n$ normalizes $\Phi$ for all $A\in\Phi$ and all $n$.

Assume that $R(f^n)=\infty$. By Corollary~\ref{R-fix}, there exists $A\in\Phi$ such that $A_*D_*^n$ has eigenvalue $1$. By \cite[Lemma~3.2]{Malfait}, $D_*^n=A_*^{-1}(A_*D_*^n)$ is virtually unipotent. Thus $D_*$ is virtually unipotent, a contradiction.
\end{proof}

\begin{Example}
Let $f$ be an Anosov diffeomorphism on an infra-nilmanifold. By \cite[Lemma~4.2]{DRPAn}, $f$ has an affine homotopy lift $(d,D)$ with hyperbolic $D_*$. From the above corollary, the Gauss congruences \eqref{DR} and \eqref{DN} hold true.
\end{Example}

\begin{Example}[{\cite[Example~11]{Fel00}, \cite{LL-JGP}}]
Let $f:M\to M$ be an expanding smooth map on a closed smooth manifold. It is known in \cite{Gromov} that $f$ is topologically conjugate to an expanding map on an infra-nilmanifold. Thus we can assume that $M$ is an infra-nilmanifold and $f$ is a map induced by an affine map $(d,D)$, where all the eigenvalues of $D_*$ are of modulus $>1$.
Since $(d,D)$ satisfies the conditions of Corollary~\ref{cor:congruence}, all $R(f^n)$ are finite and so the congruences \eqref{DR} and \eqref{DN} hold true.

On the other hand, by \cite{Shub69}, the set $\fix(f^n)$ of fixed points of the expanding map $f^n$ is nonempty and finite. Thus by Proposition~\ref{KL} and Corollary~\ref{R-fix} we have $N(f^n)=\#\fix(f^n)=R(f^n)$.
\end{Example}

\end{document}